\documentclass[12pt]{amsart}

\usepackage{amsmath, amsthm, amssymb, amsfonts, amscd, graphicx, endnotes, tikz, color}

\usepackage{mathtools}

\usepackage[margin=01.2in]{geometry}

\usetikzlibrary{arrows}

\DeclarePairedDelimiter\abs{\lvert}{\rvert}%
\DeclarePairedDelimiter\norm{\lVert}{\rVert}%

\def\CC{{\mathbb C}}
\def\FF{{\mathbb F}}

\def\NN{{\mathbb N}}

\def\QQ{{\mathbb Q}}

\def\QQ{{\mathbb Q}}
\def\RR{{\mathbb R}}

\def\ZZ{{\mathbb Z}}

\def\0{{\mathbf 0}}
\def\1{{\mathbf 1}}

\def\Acal{{\mathcal A}}
\def\Bcal{{\mathcal B}}

\def\Hcal{{\mathcal H}}

\def\Ocal{{\mathcal O}}

\def\Kbar{{\bar K}}

\def\dist{\mathrm{dist}}

\def\Per{\mathrm{Per}}

\def\I{\mathrm{I}}
\def\II{\mathrm{II}}
\def\III{\mathrm{III}}
\def\sup{\mathrm{sup}}
\def\max{\mathrm{max}}

\theoremstyle{plain}

\newtheorem{thm}{Theorem}
\newtheorem{conj}{Conjecture}
\newtheorem{cor}[thm]{Corollary}
\newtheorem{prop}[thm]{Proposition}
\newtheorem{lem}[thm]{Lemma}

\theoremstyle{definition}

\newtheorem{rem}{Remark}

\title[Non-Archimedean H\'enon maps]{Non-Archimedean H\'enon maps, \\ attractors, and horseshoes}

\author{Kenneth Allen}

\address{Kenneth Allen; Department of Mathematical Sciences; University of Massachusetts Lowell; 
Lowell MA 01854 U.S.A.}

\email{kenneth\_allen1@student.uml.edu}

\author{David DeMark}

\address{David DeMark; Department of Mathematics and Statistics; Carleton College; Northfield MN 55057 U.S.A.}

\email{demarkd@carleton.edu}

\author{Clayton Petsche}

\address{Clayton Petsche; Department of Mathematics; Oregon State University; Corvallis OR 97331 U.S.A.}

\email{petschec@math.oregonstate.edu}

\date{Latest revision January 5 2018}

\begin{document}

\begin{abstract}
We study the dynamics of the H\'enon map defined over complete, locally compact non-Archimedean fields of odd residue characteristic.  We establish basic properties of its one-sided and two-sided filled Julia sets, and we determine, for each H\'enon map, whether these sets are empty or nonempty, whether they are bounded or unbounded, and whether they are equal to the unit ball or not.  On a certain region of the parameter space we show that the filled Julia set is an attractor.  We prove that, for infinitely many distinct H\'enon maps over $\QQ_3$, this attractor is infinite and supports an SRB-type measure describing the distribution of all nearby forward orbits.  We include some numerical calculations which suggest the existence of such infinite attractors over $\QQ_5$ and $\QQ_7$ as well.  On a different region of the parameter space, we show that the H\'enon map is topologically conjugate on its filled Julia set to the two-sided shift map on the space of bisequences in two symbols.
\end{abstract}

\maketitle


\section{Introduction}

\subsection{Summary of results}  Let $K$ be a field which is complete and locally compact with respect to a nontrivial, non-Archimedean absolute value $|\cdot|$, and such that the associated residue field has odd characteristic.  Familiar examples include the field $\QQ_p$ of $p$-adic numbers for some odd prime $p$, and the fraction field $\FF_p(\!(T)\!)$ of the formal power series ring $\FF_p[\![T]\!]$ over the finite field $\FF_p$ for some odd prime $p$.  More generally, $K$ could be any finite extension of $\QQ_p$ or $\FF_p(\!(T)\!)$.

In this paper we study the dynamics of the H\'enon map $\phi_{a,b}:K^2\to K^2$ defined by
\begin{equation}\label{HenonDef}
\phi_{a,b}(x,y) = (a+by-x^2,x) \hskip1cm a,b\in K, b\neq0
\end{equation}
over such fields.  It is easily checked that $\phi_{a,b}$ is an automorphism of the plane $K^2$, with inverse $\phi_{a,b}^{-1}:K^2\to K^2$ defined by 
\begin{equation}\label{HenonInverse}
\textstyle\phi_{a,b}^{-1}(x,y) = (y,\frac{1}{b}(-a+x+y^2)).
\end{equation}

The definition $(\ref{HenonDef})$ represents one of several standard forms for the H\'enon map which are commonly found in the literature.  As described in Friedland-Milnor \cite{MR991490}, it can be shown that every quadratic polynomial automorphism of $K^2$ of dynamical degree $>1$ (in the sense of \cite{MR991490}) is affine-conjugate to a map of the form $(\ref{HenonDef})$; and moreover, no two distinct maps of the form $(\ref{HenonDef})$ are affine-conjugate to one another.  It is therefore natural to set 
\begin{equation*}
\Hcal=\{(a,b)\in K\times K\mid b\neq0\}
\end{equation*}
and to consider $\Hcal$ as the space parametrizing all distinct H\'enon maps over $K$.

For each integer $n\geq1$, denote by $\phi^n_{a,b}$ the $n$-fold composition of $\phi_{a,b}$ with itself, and by $\phi^{-n}_{a,b}$ the $n$-fold composition of $\phi^{-1}_{a,b}$ with itself.  Naturally, $\phi^0_{a,b}$ is interpreted to be the identity map on $K^2$.  Consider the three sets 
\begin{equation}\label{FilledJuliaSetsDef}
\begin{split}
J^+(\phi_{a,b}) & = \{(x,y)\in K^2\mid \|\phi^n_{a,b}(x,y)\|\text{ is bounded as }n\to+\infty\} \\
J^-(\phi_{a,b}) & = \{(x,y)\in K^2\mid \|\phi^n_{a,b}(x,y)\|\text{ is bounded as }n\to-\infty\} \\
J(\phi_{a,b}) & = J^+(\phi_{a,b})\cap J^-(\phi_{a,b}),
\end{split}
\end{equation}
the forward filled Julia set, backward filled Julia set, and (two-sided) filled Julia set, respectively.  As these boundedness properties are shared by all elements of the orbit $\{\phi_{a,b}^n(x,y)\mid n\in\ZZ\}$ of any given point $(x,y)\in K^2$, each of the three sets $J,J^\pm$ is $\phi_{a,b}$-invariant; in other words $\phi_{a,b}(J)=J$, and similarly for $J^\pm$.

The main goal of this work is to describe the filled Julia sets $J(\phi_{a,b})$ and $J^\pm(\phi_{a,b})$, to describe the dynamical behavior of the map $\phi_{a,b}$ on these sets, and to study how the dynamical properties of $\phi_{a,b}$ vary as the parameters $a$ and $b$ range over all possible values in the parameter space $\Hcal$.

In order to state our main results, we must describe a partition of the parameter space $\Hcal$ into the four regions
\begin{equation}\label{Regions}
\begin{split}
\Hcal_{\I} & = \{(a,b)\in \Hcal\mid |a|\leq 1, |b|=1\} \\
\Hcal_{\II}^+ & = \{(a,b)\in \Hcal\mid |a|\leq 1, |b|<1\} \\
\Hcal_{\II}^- & = \{(a,b)\in \Hcal\mid |a|\leq |b|^2, |b|>1\} \\
\Hcal_{\III} & = \{(a,b)\in \Hcal\mid |a|> \max(1,|b|^2)\}.
\end{split}
\end{equation}

\begin{figure}[hbt]
\begin{center}
\includegraphics[width=80mm]{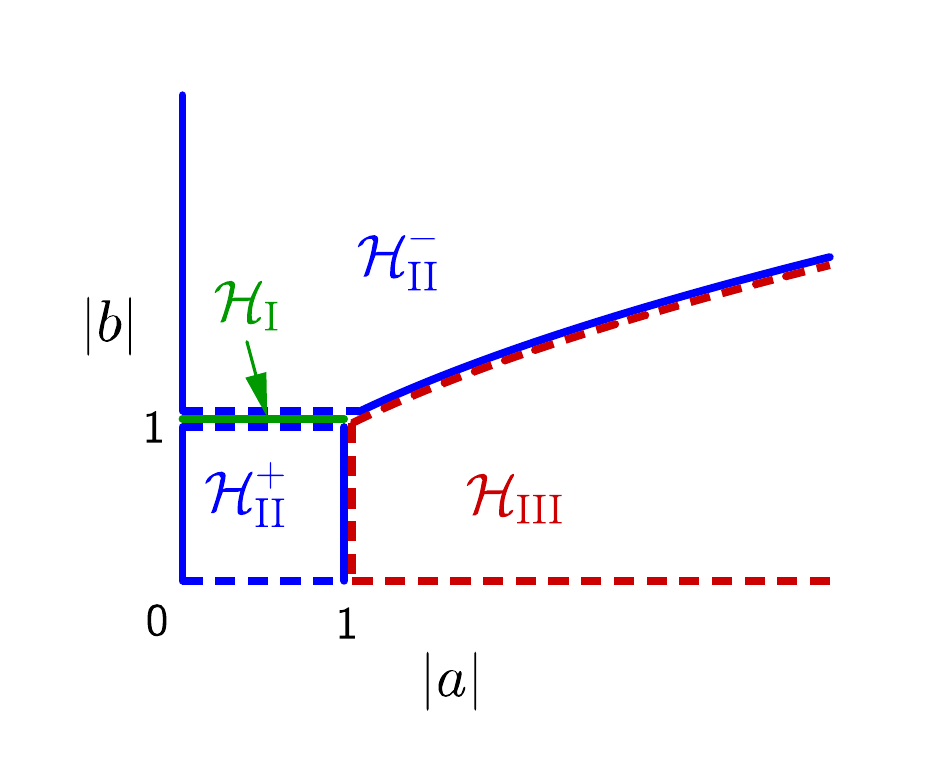}
\caption{The partition of the parameter space $\Hcal$ into the four regions $\Hcal_\I$, $\Hcal_\II^+$, $\Hcal_\II^-$, and $\Hcal_\III$, depicted in the $(|a|,|b|)$-plane.}
\label{ParameterSpaceImage}
\end{center}
\end{figure}

To give some intuition for the partition described in $(\ref{Regions})$, we first point out that the inverse of a H\'enon map $\phi_{a,b}$ in standard form $(\ref{HenonDef})$ is not a map occurring in the same form.  But as observed by Devaney-Nitecki \cite{MR539548}, $\phi_{a,b}^{-1}$ is linearly conjugate to the map  $\phi_{a/b^2,1/b}$.  This idea gives rise naturally to the involution 
\begin{equation}\label{InvolutionDef}
\textstyle\iota:\Hcal\to\Hcal \hskip1cm \iota(a,b)=(\frac{a}{b^2},\frac{1}{b})
\end{equation}
on the parameter space of all H\'enon maps; see Proposition~\ref{conj winv} for more details.  Using the involution, we may alternately characterize the partition $(\ref{Regions})$ by
\begin{equation}\label{RegionsAlt}
\begin{split}
\Hcal_{\I} & = \{(a,b)\in \Hcal\mid \|(a,b)\|\leq1,\|\iota(a,b)\|\leq1\} \\
\Hcal_{\II}^+ & = \{(a,b)\in \Hcal\mid \|(a,b)\|\leq1,\|\iota(a,b)\|>1\} \\
\Hcal_{\II}^- & = \{(a,b)\in \Hcal\mid \|(a,b)\|>1,\|\iota(a,b)\|\leq1\} \\
\Hcal_{\III} & = \{(a,b)\in \Hcal\mid \|(a,b)\|>1,\|\iota(a,b)\|>1\}.
\end{split}
\end{equation}

Next, we observe that $\Hcal_\I$ is precisely the region in which both the H\'enon map $\phi_{a,b}$ and its inverse $\phi_{a,b}^{-1}$ have coefficients in the ring $\Ocal$ of integers of $K$.  Using properties of the non-Archimedean absolute value it follows that $\phi_{a,b}$ restricts to a bijection $\phi_{a,b}:\Ocal^2\to\Ocal^2$, and reduces to a H\'enon map $\phi_{\overline{a},\overline{b}}:\FF_K^2\to\FF_K^2$ defined over the residue field $\FF_K$ of $K$.

When $(a,b)\in\Hcal_\II^+$, the map $\phi_{a,b}$ has coefficients in $\Ocal$, but $\phi_{a,b}^{-1}$ does not.  As a consequence, the restriction $\phi_{a,b}:B_1(0,0)\to B_1(0,0)$ to the closed unit ball of $K^2$ fails to be surjective, which leads to an asymmetry in the forward and backward dynamics of $\phi_{a,b}$.  For example, $J^-(\phi_{a,b})$ is bounded while $J^+(\phi_{a,b})$ is unbounded, and $J(\phi_{a,b})$ has the structure of a trapped attracting set in the sense of Milnor \cite{milnor:dynamicslectures}.

When $(a,b)\in\Hcal_\II^-$, the map $\phi_{a,b}$ does not have coefficients in $\Ocal$, but $\phi^{-1}_{a,b}$ is linearly conjugate to $\phi_{\iota(a,b)}$ for $\iota(a,b)\in\Hcal_\II^+$.  Consequently, maps in regions $\Hcal_\II^+$ and $\Hcal_\II^-$ exhibit identical dynamics but with a reversal of the (discrete) time direction.

Finally, when $(a,b)\in\Hcal_\III$, neither $\phi_{a,b}$ nor $\phi^{-1}_{a,b}$ is linearly conjugate to a H\'enon map in the form $(\ref{HenonDef})$ with coefficients in $\Ocal$.  As we shall see, this leads to rich dynamics related to the Smale horseshoe map (\cite{devaney:dynamicsbook} $\S$2.3).  Our main result for $(a,b)\in\Hcal_\III$ is a non-Archimedean analogue of a theorem on the real H\'enon map due to Devaney-Nitecki \cite{MR539548}.

We now state a theorem which summarizes the general results of this paper.  See $\S$\ref{RecurrentSect} for the definition of a recurrent point, and see $\S$\ref{StrageAttDefSect} for the definition of an attractor.

\begin{thm}\label{MainTheoremIntro}
Let $(a,b)\in\Hcal$ and let $\phi_{a,b}:K^2\to K^2$ be the associated H\'enon map defined in $(\ref{HenonDef})$.
\begin{itemize}
	\item[{\bf (a)}]  $J(\phi_{a,b})$ is compact.  It is empty if and only if $(a,b)\in\Hcal_\III$ and $a$ is not a square in $K$.
	\item[{\bf (b)}]  $(a,b)\in\Hcal_\I$ if and only if $J(\phi_{a,b})$is equal to the closed unit ball $B_1(0,0)$.
	\item[{\bf (c)}]  If $(a,b)\in\Hcal_\I\cup\Hcal_\II^+$ then $J(\phi_{a,b})$ is equal to the set of  recurrent points associated to $\phi_{a,b}$.
	\item[{\bf (d)}]  If $(a,b)\in\Hcal_\II^+$ then $J(\phi_{a,b})$ is an attractor for $\phi_{a,b}$.
	\item[{\bf (e)}]  If $(a,b)\in\Hcal_\III$ and $a$ is a square in $K$, then the restriction of $\phi_{a,b}$ to $J(\phi_{a,b})$ is topologically conjugate to the two-sided shift map on the space of bisequences in two symbols.
\end{itemize}
\end{thm}

The requirement in part {\bf (e)} that $a$ is a square in $K$ is not limiting, as this situation can always be obtained by replacing $K$ with its (at most quadratic) extension $K(\sqrt{a})$.

Since H\'enon ~\cite{MR0422932} introduced his namesake map in 1976, it has been the object of no small amount of study in the real and complex settings.  This deceptively simple family of polynomial maps has given rise to a surprising number of interesting dynamical features; see the surveys Robinson ~\cite{MR1792240} and Devaney \cite{devaney:dynamicsbook}.  

Devaney-Nitecki \cite{MR539548} proved in 1979 that, for certain values of its coefficients, the real H\'enon map is topologically conjugate to the two-sided shift map on the space of bisequences in two symbols.  Their proof relies on the fact that the real H\'enon map is essentially an algebraic manifestation of the Smale horseshoe map (\cite{devaney:dynamicsbook} $\S$2.3).  Our proof of Theorem~\ref{MainTheoremIntro} {\bf (e)} follows a fairly close parallel with the proof of Devaney-Nitcki \cite{MR539548} (see also Moser \cite{MR1829194}), but the details are rather complicated, and the analogy is not always so straightforward, as some of the analytic tools available over $\RR$ do not carry over to the non-Archimedean situation.

The existence of a strange attractor admitted by the real H\'enon map was proposed by H\'enon \cite{MR0422932} himself in 1976, and the first proof of the existence of such an attractor was given by Benedicks-Carleson \cite{MR1087346} in 1991.  Mora-Viana \cite{MR1237897} showed the existence of an infinite class of parameter values inducing a strange attractor.  

To study the attractors desribed in Theorem~\ref{MainTheoremIntro}, in $\S$~\ref{RegionsIandII+} we relate the structure of $J(\phi_{a,b})$ to periodic cycles with respect to the finite dynamical systems that $\phi_{a,b}$ induces on balls in $B_1(0,0)$, borrowing ideas from Anashin-Khrennikov \cite{MR2533085}.  We conjecture that, for each complete, locally compact non-Archimedean field with odd residue characteristic, there always exist some values of $(a,b)$ in $\Hcal_\II^+$ for which $J(\phi_{a,b})$ is an infinite set.  The following theorem verifies this conjecture over the field $\QQ_3$ of $3$-adic numbers.

\begin{thm}\label{StrangeAttractorThmIntro}
Suppose that $a\in \QQ_3$ satisfies $|a-2|_3\leq 1/9$, and define $\phi:\QQ_3^2\to\QQ_3^2$ by $\phi(x,y)=(a+3y-x^2,x)$.  The attractor $J(\phi)$ is uncountably infinite, has Haar measure zero in $\QQ_3^2$, and contains no periodic points.  Each point of $J(\phi)$ has dense forward orbit in $J(\phi)$.  There exists a probability measure $\mu_\phi$ supported on $J(\phi)$ with the property that the forward orbit of any point in $\ZZ_3^2$ is $\mu_\phi$-equidistributed.
\end{thm}

The measure $\mu_\phi$ describing the distribution of forward orbits near the attractor plays a role analogous to that of the SRB measure (Sinai-Ruelle-Bowen) originating in the theory of Anosov and Axiom A dynamical systems; see \cite{MR1933431}.

While there seems to be no consensus among dynamicists on the proper definition of a strange attractor (see for example \cite{MR2255038}), certainly many of the properties listed in the statement of Theorem~\ref{StrangeAttractorThmIntro} are typical of attractors to which the term ``strange'' is generally applied.  On the other hand, unlike the real H\'enon map, the dynamics described in Theorem~\ref{StrangeAttractorThmIntro} are not chaotic: for any $(a,b)$ in region $\Hcal_\II^+$ of parameter space, the H\'enon map $\phi_{a,b}$ is nonexpanding on $J(\phi_{a,b})$ and hence the forward orbits of nearby points do not diverge from one another.  Moreover, inspection of the proof of Theorem~\ref{StrangeAttractorThmIntro} shows that the Hausdorff dimension of the attractor $J(\phi)$ is $1$. 

In $\S$~\ref{NumericalCalc} we include some numerical calculations which suggest that attractors similar to the one described in Theorem~\ref{StrangeAttractorThmIntro} also exist over $\QQ_5$ and $\QQ_7$.  A proof that such attractors exist for all odd primes $p$ would be extremely interesting and a significant advance over the somewhat ad hoc proof of Theorem~\ref{StrangeAttractorThmIntro}.  In $\S$~\ref{StrangeAttSection} we also give an infinite family of distinct $3$-adic H\'enon maps for which the attractor $J(\phi_{a,b})$ is finite (in fact an attracting $2$-cycle).  

Some researchers have studied the arithmetic aspects of the H\'enon map and more general plane polynomial automorphisms, e.g. Silverman ~\cite{silverman:henonmap}, Denis \cite{MR1355926}, Ingram ~\cite{MR3180596} and others.  The purely local dynamics of the H\'enon map over non-Archimedean fields has been relatively neglected; but see Marcello \cite{MR1988948}.  Woodcock-Smart \cite{MR1678087} and Arrowsmith-Vivaldi \cite{MR1218804} have studied $p$-adic automorphisms related to horseshoe dynamics, but for different maps than the standard H\'enon map $(\ref{HenonDef})$.  One of the inspirations behind our work is the paper of Benedetto-Briend-Perdry \cite{MR2394889}, which is an analogous study in the somewhat simpler (non-invertible) setting of quadratic polynomial maps in one non-Archimedean variable.  

Our standing assumption of odd residue characteristic unifies our exposition and simplifies many of our proofs.  But the residue characteristic $2$ case would also be very interesting to consider, and much of this work would apply with suitable modifications.

We have chosen to focus on the setting of a locally compact ground field, but it would also be interesting to consider the dynamics of H\'enon maps over a complete and algebraically closed non-Archimedean field, both on the classical affine plane, and in the sense of Berkovich.  Some of our results should carry over easily; for example, the criterion for good reduction stated in Theorem~\ref{MainTheoremIntro} {\bf (b)} should still hold with basically the same proof.  We also point out that, in the setting of Theorem~\ref{MainTheoremIntro} {\bf (e)}, the topological conjugacy with the shift map implies that extending the ground field $K$ does not produce any new points in the filled Julia set.


\subsection{Plan of the paper}  In $\S$~\ref{PrelimSection} we establish some notation and terminology and discuss properties of non-Archimedean fields.  In $\S$~\ref{GeneralPopSection} we study fixed points and $2$-cycles, and following Devaney-Nitecki ~\cite{MR539548} and Bedford-Smillie ~\cite{MR1653043} we establish a filtration property satisfied by the H\'enon map.  Using this filtration we investigate basic topological and set-theoretic properties of the filled Julia sets $J^\pm(\phi_{a,b})$ and $J(\phi_{a,b})$ as $(a,b)$ ranges over the four regions of the parameter space.  In $\S$~\ref{RegionsIandII+} we consider the case $(a,b)\in\Hcal_\I\cup\Hcal_\II^+$, and we study the dynamics of the restriction $\phi_{a,b}:B_1(0,0)\to B_1(0,0)$ to the closed unit ball of $K^2$, leading to a proof that $J(\phi_{a,b})$ is equal to the recurrent set of $\phi_{a,b}$.  Restricting to the case $(a,b)\in\Hcal_\II^+$ we prove that $J(\phi_{a,b})$ is an attractor, and we study the question of whether or not this attractor is a finite union of attracting cycles.  Over $\QQ_3$, we prove that both situations can occur for infinitely many distinct H\'enon maps.  Finally, in $\S$~\ref{RegionIII} we prove Theorem~\ref{MainTheoremIntro} {\bf (e)}, establishing horseshoe dynamics in region $\Hcal_\III$.


\subsection{Acknowledgements}  This work was done during the Summer 2016 REU program in mathematics at Oregon State University, with support by National Science Foundation Grant DMS-1359173.  We thank the anonymous referees for many helpful suggestions, and simplifications of the proofs of Theorem~\ref{StrangeAttractorThm} and Theorem~\ref{TopologicalConjThm}.


\section{Non-Archimedean fields}\label{PrelimSection}

Throughout this paper, $K$ is a field which is complete and locally compact with respect to a nontrivial, non-Archimedean absolute value $|\cdot|$, and such that the associated residue field has odd characteristic.  In particular, we have the strong triangle inequality $|x+y|\leq\max(|x|,|y|)$ for all $x,y\in K$, and a standard argument (\cite{MR1488696} Prop. 2.3.3) shows that $|x+y|=\max(|x|,|y|)$ whenever $|x|\neq |y|$.  Our assumption that $K$ has odd residue characteristic can be characterized by the assumption that $|2|=1$.  

We denote by $\Ocal =\{x\in K \mid |a|\leq1\}$ the ring of integers of $K$, we let $\pi\in\Ocal$ be a uniformizing parameter.  Let $\FF_K=\Ocal/\pi\Ocal$ denote the residue field and let $x\mapsto \overline{x}$ denote the reduction map $\Ocal\to\FF_K$.  It follows from the local compactness of $K$ that a uniformizing parameter exists and that the residue field is finite.  By assumption, $\FF_K$ has odd characteristic.

Given $c\in K$ and an element $r\in|K^\times|$ of the value group of $K$, we define the ``closed'' and ``open'' discs with center $c$ and radius $r$ by
\begin{equation}\label{Discs}
\begin{split}
D_r(c) & = \{x\in K \mid |x-c|\leq r\} \\
D_r^\circ(c) & = \{x\in K \mid |x-c|<r\}.
\end{split}
\end{equation}

We use the non-Archimedean norm $\|\cdot\|$ on $K^2$ defined by $\|(x,y)\|=\max(|x|,|y|)$.  Given a point $(c_1,c_2)\in K^2$ and an element $r\in|K^\times|$ of the value group of $K$, we define the ``closed'' and ``open'' balls with center $(c_1,c_2)$ and radius $r$ by
\begin{equation}\label{Balls}
\begin{split}
B_r(c_1,c_2) & = \{(x,y)\in K^2 \mid \|(x,y)-(c_1,c_2)\|\leq r\} \\
B_r^\circ(c_1,c_2) & = \{(x,y)\in K^2 \mid \|(x,y)-(c_1,c_2)\|< r\}.
\end{split}
\end{equation}
Given two elements $r_1,r_2\in|K^\times|$ of the value group, denote by 
\begin{equation}\label{Polydiscs}
D_{r_1,r_2}(c_1,c_2) = \{(x,y)\in K^2 \mid |x-c_1|\leq r_1, |x-c_2|\leq r_2\}
\end{equation}
the polydisc in $K^2$ with center $(c_2,c_2)$ and radii $r_1,r_2$.  Of course, properties of the non-Archimedean topology on $K$ show that all of the sets defined in $(\ref{Discs})$, $(\ref{Balls})$, and $(\ref{Polydiscs})$ are topologically both open and closed in $K$.


\section{General properties of the H\'enon map}\label{GeneralPopSection}

\subsection{Fixed points and $2$-cycles}  In this section we describe criteria for the existence of fixed points and $2$-cycles.  The proof of the following proposition follows from straightforward calculations arising from the equations $\phi_{a,b}(x,y)=(x,y)$ and $\phi_{a,b}^2(x,y)=(x,y)$.  We let $\Kbar$ denote an algebraic closure of $K$.  

\begin{prop}\label{fp criteria}
Let $(a,b)\in\Hcal$ and let $\phi=\phi_{a,b}:K^2\to K^2$ be the associated H\'enon map.
\begin{itemize}
	\item[{\bf (a)}]  The only fixed points of $\phi$ in $\Kbar^2$ are $(\alpha,\alpha)$, where $\alpha$ is a root of 
$$
x^2-(b-1)x-a=0.
$$
In particular, $\phi$ has no fixed points in $K^2$ when $(b-1)^2+4a$ is not a square in $K$, a single fixed point in $K^2$ when $(b-1)^2+4a=0$, and two distinct fixed points in $K^2$ when $(b-1)^2+4a$ is a nonzero square in $K$.
	\item[{\bf (b)}]  The only solutions in $\Kbar^2$ to the equation $\phi^2(x,y)=(x,y)$ are the fixed points described in part {\bf (a)}, and the points $(\beta_1,\beta_2)$ and $(\beta_2,\beta_1)$, where $\beta_1,\beta_2$ are the roots of
$$
x^2+(b-1)x+(b-1)^2-a=0.
$$  
In particular, $\phi$ has no $2$-cycles in $K^2$ when $4a-3(b-1)^2$ is not a square in $K$, and $\phi$ has one $2$-cycle in $K^2$ when $4a-3(b-1)^2$ is a nonzero square in $K$.  When $4a-3(b-1)^2=0$, $\phi$ has fixed points $(\frac{1}{2}(1-b),\frac{1}{2}(1-b))$ and $(\frac{3}{2}(b-1),\frac{3}{2}(b-1))$ but no $2$-cycles in $K^2$.
\end{itemize}
\end{prop}

The following proposition gives a selection of sample applications of Proposition~\ref{fp criteria}, showing that periodic points and $2$-cycles can be easily obtained for certain subregions of the parameter space $\Hcal$.

\begin{prop}\label{PeriodicIIandIII} Let $(a,b)\in\Hcal$ and let $\phi=\phi_{a,b}:K^2\to K^2$ be the associated H\'enon map.
\begin{itemize}
	\item[{\bf (a)}] Suppose that $|a|<1$, $|b|\leq1$, and $\overline{b}\neq1$ in $\FF_K$.  Then $\phi$ has two distinct fixed points in $K^2$.  Moreover, if $K$ has residue characteristic not equal to $3$ and if $-3$ is a square in $K$, then $\phi$ has a $2$-cycle in $K^2$.
	\item[{\bf (b)}] Suppose that $(a,b)\in\Hcal_\II^-$ and that $|a|<|b|^2$.  Then $\phi$ has two distinct fixed points in $K^2$.  Moreover, if $K$ has residue characteristic not equal to $3$ and if $-3$ is a square in $K$, then $\phi$ has a $2$-cycle in $K^2$.
	\item[{\bf (c)}]  Suppose that $(a,b)\in\Hcal_\III$ and that $a$ is a square in $K$.  Then $\phi$ has two distinct fixed points in $K^2$ and a $2$-cycle in $K^2$.
\end{itemize}
\end{prop}

\begin{proof}
We prove {\bf (a)}.  Substituting $x=(b-1)t$ in the fixed-point and two-cycle equations leads to $t^2-t-a(b-1)^{-2}=0$ and $t^2+t+(1-a(b-1)^{-2})=0$.  The former has roots in $K$ by inspection of its Newton polygon, and under the additional hypotheses, $K$ contains a primitive $3$-rd root of unity, and hence the latter has roots in $K$ by Krasner's lemma.  The proofs of {\bf (b)} and {\bf (c)} are similar, using the substitutions $x=(b-1)t$ and $x=t\sqrt{a}$, respectively.
\end{proof}

In fact, much more than Proposition~\ref{PeriodicIIandIII} part {\bf (c)} is true.  When $(a,b)$ is in $\Hcal_\III$ and $a$ is a square in $K$, it follows from the results of $\S$~\ref{RegionIII} that $\phi_{a,b}$ has $2^\ell$ distinct points of period $\ell$ in $K^2$ for each $\ell\geq1$, that all possible minimal periods occur in $K^2$, and that all $\phi_{a,b}$-periodic points in $\Kbar^2$ occur in $K^2$.  On the other hand, the proof of Proposition~\ref{PeriodicIIandIII} part {\bf (c)} is fairly elementary and does not rely on the machinery of $\S$~\ref{RegionIII}.

We also remark that the following converse of Proposition~\ref{PeriodicIIandIII} part {\bf (c)} is true: if $(a,b)$ is in $\Hcal_\III$ and $a$ is not a square in $K$, then $\phi_{a,b}$ has no periodic points at all in $K^2$.  The non-existence of fixed points and $2$-cycles could be obtained easily from Proposition~\ref{fp criteria}, but the non-existence of any periodic points whatsoever follows from Theorem~\ref{JuliaSetBoundedNonempty} part {\bf (d)}, as filled Julia sets contain all periodic points.  


\subsection{An involution of the parameter space}  In this section we record a proposition summarizing the basic properties of the involution $\iota:\Hcal\to\Hcal$ of the parameter space described in the introduction.  We omit the proof which consists of routine calculations.

\begin{prop}\label{conj winv}
The function $\iota:\Hcal\to\Hcal$ defined by $\iota(a,b)=(\frac{a}{b^2},\frac{1}{b})$ is an involution, it restricts to involutions $\iota:\Hcal_\I\to\Hcal_\I$ and $\iota:\Hcal_{\III}\to\Hcal_{\III}$, and it restricts to bijections $\iota:\Hcal_{\II}^+\to\Hcal_{\II}^-$ and $\iota:\Hcal_{\II}^-\to\Hcal_{\II}^+$.  Given $(a,b)\in\Hcal$, the maps $\phi_{\iota(a,b)}$ and $\phi_{a,b}^{-1}$ are linearly conjugate to one another by the automorphism $\lambda:K^2\to K^2$ defined by $\lambda(x,y)=(-by,-bx)$; more precisely, $\phi_{\iota(a,b)}=\lambda^{-1}\circ \phi_{a,b}^{-1}\circ\lambda$.  
\end{prop}


\subsection{A filtration satisfied by the H\'enon map}\label{FiltrationSect}  For each $(a,b)\in\Hcal$, set
\begin{equation}\label{RDef}
R=R_{a,b}=\max(1,|a|^{1/2},|b|)
\end{equation}
and divide the plane $K^2$ into the three subsets
\begin{equation}\label{FiltrationSetsDef}
\begin{split}
S_{R} & = \{(x,y)\in K^2 \mid \|(x,y)\|\leq R \} \\
S_{R}^+ & =\{(x,y)\in K^2 \mid \|(x,y)\|> R, |x|\geq|y|\} \\
S_{R}^- & =\{(x,y)\in K^2 \mid \|(x,y)\|> R, |x|\leq|y|\}.
\end{split}
\end{equation}
Clearly $K^2= S_{R}\cup S_{R}^+\cup S_{R}^-$, but this is not technically a partition: $S_{R}$ intersects neither $S_{R}^+$ nor $S_{R}^-$, but $S_{R}^+$ intersects $S_{R}^-$ where $|x|=|y|>R$. 

\begin{figure}[hbt]
\begin{center}
\includegraphics[width=80mm]{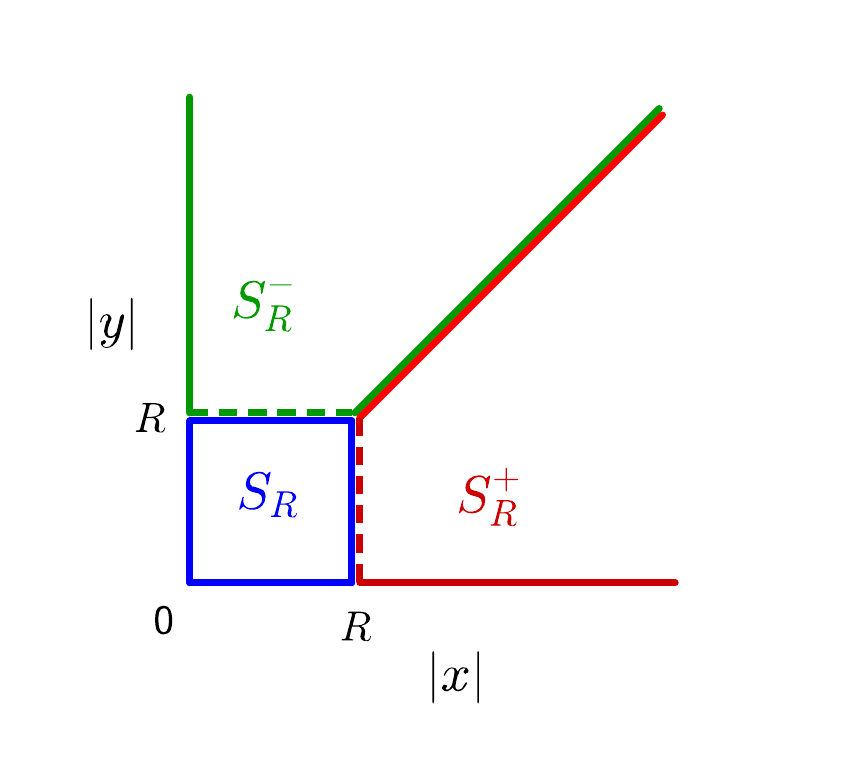}
\caption{The filtration sets $S_{R}$, $S_{R}^+$, and $S_{R}^-$ depicted in the $(|x|,|y|)$-plane.}
\label{FiltrationImage}
\end{center}
\end{figure} 

Following the trajectories of $\phi_{a,b}$-orbits though the sets $(\ref{FiltrationSetsDef})$ leads to the following useful facts: every point in $K^2$ either has a bounded forward orbit eventually contained in $S_R$, or else its forward orbit eventually gets filtered through $S_{R}^+$ to infinity.  Similarly, every point in $K^2$ either has a bounded backward orbit eventually contained in $S_R$, or else its backward orbit eventually gets filtered through $S_{R}^-$ to infinity.  We describe this filtration in more detail in Proposition~\ref{Filtration}, which has been adapted to our purposes from Bedford-Smillie \cite{MR1653043}, who give a treatment of such a filtration for more general plane polynomial automorphisms over $\CC$.  A similar filtration is used by Devaney-Nitecki \cite{MR539548} in their work on the real H\'enon map.

Before we state Proposition~\ref{Filtration}, we require the following lemma, which shows how the parameter $R$ and the sets $S_R$ and $S_R^\pm$ interact with the involution $\iota:\Hcal\to\Hcal$ described in Proposition~\ref{conj winv}.  The proof consists of elementary calculations which we omit.

\begin{lem}\label{FiltrationandInvolution} 
Let $(a,b)\in\Hcal$, and define $R$, $S_R$, $S_R^+$, and $S_R^-$ as in $(\ref{RDef})$ and $(\ref{FiltrationSetsDef})$.  Let $\iota(a,b)=(\frac{a}{b^2},\frac{1}{b})=(a^*,b^*)$, and similarly set $R^*=\max(1,|a^*|^{1/2},|b^*|)$ and define $S_{R^*}$, $S_{R^*}^+$, and $S_{R^*}^-$ as in $(\ref{FiltrationSetsDef})$.  Define $\lambda:K^2\to K^2$ by $\lambda(x,y)=(-by,-bx)$.  The following identities hold.
\begin{equation*}
\begin{split}
 R^* & = \textstyle\frac{1}{|b|}R \\
S_{R^*} & = \lambda^{-1}(S_R) \\
S_{R^*}^+ & = \lambda^{-1}(S_R^-) \\
S_{R^*}^- & = \lambda^{-1}(S_R^+)
\end{split}
\end{equation*}
\end{lem}

\begin{prop}\label{Filtration} 
Let $(a,b)\in\Hcal$ and let $\phi=\phi_{a,b}:K^2\to K^2$ be the associated H\'enon map.  Set $R=\max(1,|a|^{1/2},|b|)$, and define $S_R$, $S_R^+$, and $S_R^-$ as in $(\ref{FiltrationSetsDef})$.  Then $\phi$ satisfies the following properties related to forward iteration.
\begin{itemize}
	\item[{\bf (a)}] If $(x,y)\in S_{R}^+$, then $\phi(x,y)\in S_{R}^+$ and $\lim_{n\to +\infty}\norm{\phi^{n}(x,y)}=+\infty$.
	\item[{\bf (b)}] If $(x,y)\in S_{R}$, then $\phi(x,y)\notin S_{R}^-$
	\item[{\bf (c)}] If $(x,y)\in S_{R}^-$, then $\phi^n(x,y)\in S_{R}^-$ for only finitely many $n\geq0$.
\end{itemize}
Moreover, $\phi$ satisfies the following properties related to backward iteration.
\begin{itemize}
	\item[{\bf (d)}] If $(x,y)\in S_{R}^-$, then $\phi^{-1}(x,y)\in S_{R}^-$ and $\lim_{n\to +\infty}\norm{\phi^{-n}(x,y)}=+\infty$
	\item[{\bf (e)}] If $(x,y)\in S_{R}$, then $\phi^{-1}(x,y)\notin S_{R}^+$
	\item[{\bf (f)}] If $(x,y)\in S_{R}^+$, then $\phi^{-n}(x,y)\in S_{R}^+$ for only finitely many $n\geq0$.
\end{itemize}
\end{prop}

\begin{proof}
In this proof we use the shorthand notation $(x_n,y_n)=\phi^n(x,y)$ for the $n$-th iterate of the point $(x,y)=(x_0,y_0)$ under $\phi$, where $n\in\ZZ$.

To prove {\bf (a)}, suppose that $(x_0,y_0)\in S_R^+$.  Note that $|x_0|>1$, $|x_0|>|b|$, $|x_0|\geq|y_0|$, and $|x_0^2|>|a|$ by assumption, so 
$$
|x_{1}|=|a+by_0-x_0^2|=|x_0^2|>|x_0|=|y_1|.
$$ 
Thus, $(x_1,y_1)\in S_R^+$.  Iterating this argument we see that $|x_{n+1}|=|x_n|^2>|x_n|=|y_{n+1}|$ for all $n\geq1$.  Thus $\|(x_n,y_n)\|\to+\infty$ as $n\to+\infty$.

To prove {\bf (b)}, suppose that $(x_0,y_0)\in S_R$.  If $(x_1,y_1)\in S_R$ then $(x_1,y_1)\notin S_R^-$ and we are done.  If $(x_1,y_1)\notin S_R$ then $\|(x_1,y_1)\|>R$, but as $|y_1|=|x_0|\leq R$, we must have $|x_1|>R$.  Thus $|x_1|>|y_1|$ which implies $(x_1,y_1)\notin S_R^-$.

To prove {\bf (c)}, suppose that $(x_0,y_0)\in S_R^-$ and consider the following exhaustive (but not mutually exclusive) set of cases:
\begin{itemize}
	\item[] Case 1: Some forward iterate $(x_n,y_n)$ is in $S_R^+$
	\item[] Case 2: Some forward iterate $(x_n,y_n)$ is in $S_R$
	\item[] Case 3: All forward iterates $(x_n,y_n)$ satisfy $|y_n|>|x_n|$.
\end{itemize}
In Case 1, it follows from the proof of part {\bf (a)} that $|x_m|>|y_m|$, and hence $(x_m,y_m)\notin S_R^-$, for all $m>n$, and we are done.  In Case 2, {\bf (b)} implies that $(x_{n+1},y_{n+1})$ is either in $S_R$ or $S_R^+$; the latter situation puts us in Case 1, which has been settled,  so we may assume that $(x_{n+1},y_{n+1})\in S_R$.  Iterating, we must have that $(x_{m},y_{m})$ is in $S_R$, and therefore not in $S_R^-$, for all $m\geq n$, and again we are done.  Finally, in Case 3 we have $\|(x_{n+1},y_{n+1})\|=|y_{n+1}|=|x_n|<|y_n|=\|(x_{n},y_{n})\|$ for all $n\geq0$, and thus $\{\|(x_{n},y_{n})\|\}$ is a strictly decreasing sequence as $n\to+\infty$.  Since $K$ is discretely valued, this means that $(x_{n},y_{n})$ is in $S_R$, and therefore not $S_R^-$, for all large enough $n$.

Parts {\bf (d)}, {\bf (e)}, and {\bf (f)} follow by using Lemma~\ref{conj winv} and  Lemma~\ref{FiltrationandInvolution} to reformulate them as statements about $\phi_{\iota(a,b)}$, and applying parts {\bf (a)}, {\bf (b)}, and {\bf (c)} to $\phi_{\iota(a,b)}$.  For example, to prove {\bf (d)} assume that $(x,y)\in S_R^-$.  Adopting the notation of Lemma~\ref{FiltrationandInvolution}, we have $\lambda^{-1}(x,y)\in \lambda^{-1}(S_R^-)=S_{R^*}^+$.  Using {\bf (a)}, we have that $\phi_{\iota(a,b)}(\lambda^{-1}(x,y))\in S_{R^*}^+$ and therefore
$$
\phi_{a,b}^{-1}(x,y) = \lambda\circ\phi_{\iota(a,b)}\circ \lambda^{-1}(x,y)\in \lambda(S_{R^*}^+) = S_R^-,
$$
which is the desired inclusion.  The second assertion of {\bf (d)}, as well as {\bf (e)} and {\bf (f)}, follow analogously.
\end{proof}

According to the filtration described in Proposition~\ref{Filtration}, $\phi_{a,b}(S_R)$ does not intersect $S_R^-$, and therefore either $\phi_{a,b}(S_R)\subseteq S_R$ or else $\phi_{a,b}(S_R)$ intersects  $S_R^+$.   Similarly, $\phi_{a,b}^{-1}(S_R)$ does not intersect $S_R^+$, and therefore either $\phi_{a,b}^{-1}(S_R)\subseteq S_R$ or else $\phi_{a,b}^{-1}(S_R)$ intersects $S_R^-$.   The following proposition establishes precisely when each of these cases occurs.

\begin{prop}\label{Filtration2} 
Let $(a,b)\in\Hcal$ and let $\phi=\phi_{a,b}:K^2\to K^2$ be the associated H\'enon map.  Set $R=\max(1,|a|^{1/2},|b|)$, and define $S_R$, $S_R^+$, and $S_R^-$ as in $(\ref{FiltrationSetsDef})$.  
\begin{itemize}
  	\item[{\bf (a)}] $\phi(S_R)\subseteq S_R$ if and only if $\|(a,b)\|\leq 1$ (or equivalently, $(a,b)\in \Hcal_\I\cup\Hcal_\II^+$).
  	\item[{\bf (b)}] $\phi^{-1}(S_R)\subseteq S_R$ if and only if $\|\iota(a,b)\|\leq 1$ (or equivalently, $(a,b)\in \Hcal_\I\cup\Hcal_\II^-$).
\end{itemize}
\end{prop}

\begin{proof}[Proof of Proposition~\ref{Filtration2}]
{\bf (a)}  If $(a,b)$ is in region $\Hcal_\I$ or $\Hcal_\II^+$, we then have that $|a|\leq1$ and $|b|\leq1$, so $R=1$ and $S_R=B_1(0,0)$.  Since in this case $\phi$ has coefficients in $\Ocal$, we deduce that $\phi(S_R)\subseteq S_R$ by the strong triangle inequality.  Conversely, assume that $(a,b)$ is not in region $\Hcal_\I$ or $\Hcal_\II^+$; this means that either $|a|>1$ or $|b|>1$.  We split into two cases. If $\abs{a}^{1/2}\geq \abs{b}$, then $\abs{a}>1$ and $R=|a|^{1/2}$, so $\abs{a}>R$.  It follows that $\phi(0,0)=(a,0)\in S_R^+$, and therefore $\phi(S_R)\not\subseteq S_R$. If $\abs{a}^{1/2}< \abs{b}$, then $\abs{b}>1$ and $R=|b|$.  It follows that $(b,0)\in S_R$ and $|a-b^2|=\abs{b^2}>|b|=R$, so $\phi(b,0)=(a-b^2,b)\in S_R^+$.  Thus $\phi(S_R)\not\subseteq S_R$ in this case as well.

{\bf (b)}  We adopt the notation of Lemma~\ref{FiltrationandInvolution}.  It follows from Proposition~\ref{conj winv} and  Lemma~\ref{FiltrationandInvolution} that $\lambda^{-1}(S_R)=S_{R^*}$ and $\lambda^{-1}(\phi_{a,b}^{-1}(S_R))=\phi_{a^*,b^*}(S_{R^*})$, and consequently $\phi_{a,b}^{-1}(S_R)\subseteq S_R$ if and only if $\phi_{a^*,b^*}(S_{R^*})\subseteq S_{R^*}$.   By Proposition~\ref{conj winv}, $(a,b)$ is in region $\Hcal_\I$ or $\Hcal_\II^-$ if and only if $(a^*,b^*)$ is in region $\Hcal_\I$ or $\Hcal_\II^+$.  Together these facts imply that {\bf (b)} follows from {\bf (a)} applied to $\phi_{a^*,b^*}$.
\end{proof}

\begin{prop}\label{J+,J-}
Let $(a,b)\in\Hcal$ and let $\phi=\phi_{a,b}:K^2\to K^2$ be the associated H\'enon map.  Set $R=\max(1,|a|^{1/2},|b|)$, and define $S_R$, $S_R^+$, and $S_R^-$ as in $(\ref{FiltrationSetsDef})$.  Then the following holds.
\begin{equation*}
\begin{split}
J^+(\phi) & = K^2\setminus \bigg(\bigcup_{n\geq0}\phi^{-n}(S_R^+)\bigg) \subseteq S_R^-\cup S_R \\
J^-(\phi) & =K^2\setminus \bigg(\bigcup_{n\geq0}\phi^{n}(S_R^-)\bigg)\subseteq S_R^+\cup S_R \\
J(\phi) & \subseteq S_R.
\end{split}
\end{equation*}
In particular, $J(\phi)$ is bounded.
\end{prop}
\begin{proof}
Proposition~\ref{Filtration} {\bf (a)} shows that no point in $S_R^+$ can have bounded forward orbit, from which $J^+(\phi) \subseteq S_R^-\cup S_R$ follows.  More generally, if $(x,y)\in J^+(\phi)$, then each point $\phi^n(x,y)$ in its forward orbit also has bounded forward orbit, from which we deduce $\phi^n(x,y)\not\in S_R^+$; thus $(x,y)\not\in\bigcup_{n\geq0}\phi^{-n}(S_R^+)$.   Conversely, if $(x,y)\not\in\bigcup_{n\geq0}\phi^{-n}(S_R^+)$, then the points $\phi^n(x,y)$ (for $n\geq0$) are not in $S_R^+$, and can only be in $S_R^-$ for finitely many $n$ by Proposition~\ref{Filtration} {\bf (c)}.  Hence $\phi^n(x,y)\in S_R$ for all but finitely many $n\geq0$, and we deduce that $(x,y)$ has bounded forward orbit, and so it is in $J^+(\phi)$.

The corresponding facts about $J^-(\phi)$ follow from a similar argument.  Finally, since points with bounded two-sided orbit can occur neither in $S_R^+$ nor in $S_R^-$, they can only occur in $S_R$, from which $J(\phi) \subseteq S_R$ follows.
\end{proof}


\subsection{General properties of the filled Julia sets}  In this section we use the filtration described in $\S$~\ref{FiltrationSect} to deduce some basic facts about the filled Julia sets defined in $(\ref{FilledJuliaSetsDef})$.  We begin with a lemma which explains how the filled Julia sets interact with the involution $\iota:\Hcal\to\Hcal$ described in Proposition~\ref{conj winv}.  

\begin{lem}\label{FilledJuliaAndInvolution}
Let $(a,b)\in\Hcal$ and define $\lambda:K^2\to K^2$ by $\lambda(x,y)=(-by,-bx)$.  Then
\begin{equation*}
\begin{split}
J(\phi_{\iota(a,b)}) & =\lambda^{-1}(J(\phi_{a,b})) \\
J^+(\phi_{\iota(a,b)}) & =\lambda^{-1}(J^-(\phi_{a,b})) \\
J^-(\phi_{\iota(a,b)}) & =\lambda^{-1}(J^+(\phi_{a,b})).
\end{split}
\end{equation*}
\end{lem}
\begin{proof}
For each $n\in\ZZ$, it follows from Proposition~\ref{conj winv} that $\lambda\circ\phi_{\iota(a,b)}^n=\phi_{a,b}^{-n}\circ \lambda$.  Together with the easily checked identity $\|\lambda(x,y)\|=|b|\|(x,y)\|$, we deduce that 
\begin{equation}
|b|\|\phi_{\iota(a,b)}^n(x,y)\|=\|\phi_{a,b}^{-n}(\lambda(x,y))\|
\end{equation}
for each $(x,y)\in K^2$.  It follows that $(x,y)$ has bounded forward (resp. backward) $\phi_{\iota(a,b)}$-orbit if and only if $\lambda(x,y)$ has bounded backward (resp. forward) $\phi_{a,b}$-orbit, from which the desired identities follow.
\end{proof}

Our next result shows that, when $(a,b)\in\Hcal_\I$, the filled Julia sets take a particularly simple form; and in fact these conditions characterize the region $\Hcal_\I$.

\begin{thm}\label{GoodReductionCrit}
Let $(a,b)\in\Hcal$ and let $\phi=\phi_{a,b}:K^2\to K^2$ be the associated H\'enon map.  The following conditions are equivalent:
\begin{enumerate}
\item[{\bf (a)}] $(a,b)\in\Hcal_\I$
\item[{\bf (b)}] $J(\phi)=B_1(0,0)$
\item[{\bf (c)}] $J^+(\phi)=B_1(0,0)$
\item[{\bf (d)}] $J^-(\phi)=B_1(0,0)$
\item[{\bf (e)}] $\phi(B_1(0,0))=B_1(0,0)$
\end{enumerate}
\end{thm}
\begin{proof}
{\bf (a)} $\Rightarrow$ {\bf (b)}:   Suppose that  $(a,b)\in\Hcal_\I$; thus $|a|\leq 1$ and $|b|=1$.  We then have $R=1$ in Corollary~\ref{J+,J-} and hence $J(\phi) \subseteq S_1=B_1(0,0)$.  Conversely, suppose $(x,y)\in B_1(0,0)$.  Then $\norm{\phi(x,y)} = \max\{\abs{a+by-x^2},\abs{x}\}\leq 1$, and iterating we have $\norm{\phi^n(x,y)}\leq 1$ for all $n\geq0$.  Similarly, $\norm{\phi^{-1}(x,y)} = \max\{\abs{y},\abs{-\frac{a}{b}+\frac{1}{b}x+\frac{1}{b}y^2}\} \leq 1$, and iterating we have $\norm{\phi^{-n}(x,y)} \leq 1$ for all $n\geq1$.  We conclude that $(x,y) \in J(\phi)$.

{\bf (b)} $\Rightarrow$ {\bf (e)}:   If $J(\phi)=B_1(0,0)$, then $\phi(B_1(0,0))=B_1(0,0)$ follows from the $\phi$-invariance of $J(\phi)$.

{\bf (e)} $\Rightarrow$ {\bf (a)}:  Suppose that $\phi(B_1(0,0))=B_1(0,0)$. Then because $(a,0) =\phi(0,0) \in B_1(0,0)$, it must be true that $|a| \leq 1$.  Also, $(a+b,0)=\phi(0,1)  \in B_1(0,0)$, so it must be true that $\abs{a+b} \leq 1$.  Therefore $|b|\leq \max(|a+b|,|-a|)\leq1$.  Finally, note that $\phi^{-1}(B_1(0,0))=B_1(0,0)$.  So $(0,\frac{1}{b})=\phi^{-1}(a+1,0)\in B_1(0,0)$, so we must have $|\frac{1}{b}|\leq1$ and thus $|b|=1$.  

{\bf (a)} $\Rightarrow$ {\bf (c)}:   Suppose that  $(a,b)\in\Hcal_\I$.  Then $R=1$ and $S_R=B_1(0,0)$ in Proposition~\ref{J+,J-}.  Since we have already proved that $J(\phi)=B_1(0,0)$ whenever $(a,b)\in\Hcal_\I$, we have $B_1(0,0)\subseteq J^+(\phi)$.  If $J^+(\phi)$ contains some point $(x_0,y_0)$ which is not in $B_1(0,0)$, then Proposition~\ref{J+,J-} implies that $(x_0,y_0)\in S_R^-$.  By Proposition~\ref{Filtration} part {\bf (c)} and Proposition~\ref{J+,J-}, there exists some point $(x_{n},y_{n})$ in the forward orbit of $(x_0,y_0)$ such that $(x_{n},y_{n})\in S_R$.  In particular, $(x_{0},y_{0})\notin B_1(0,0)$ and $(x_{n},y_{n})\in B_1(0,0)$.  This is impossible, because we have already proved that $\phi(B_1(0,0))=B_1(0,0)$.

{\bf (a)} $\Rightarrow$ {\bf (d)}:   This is identical to the proof of {\bf (a)} $\Rightarrow$ {\bf (c)} except with the direction of iteration reversed.

{\bf (c)} $\Rightarrow$ {\bf (e)} and {\bf (d)} $\Rightarrow$ {\bf (e)}: These are the same as the proof of {\bf (b)} $\Rightarrow$ {\bf (e)}.  Because the filled Julia sets $J^+(\phi)$ and $J^-(\phi)$ are $\phi$-invariant, $\phi(B_1(0,0))=B_1(0,0)$ follows if either $J^+(\phi)=B_1(0,0)$
 or $J^-(\phi)=B_1(0,0)$.
\end{proof}

Theorem~\ref{GoodReductionCrit} gives complete information about the filled Julia sets in the region $\Hcal_\I$.  The following proposition explains, for the other three regions $\Hcal_\II^+$, $\Hcal_\II^-$, and $\Hcal_\III$, whether the filled Julia sets are empty or nonempty, bounded or unbounded, and under what conditions it occurs that $J^\pm(\phi_{a,b})=J(\phi_{a,b})$.

\begin{thm}\label{JuliaSetBoundedNonempty} Let $(a,b)\in\Hcal$ and let $\phi=\phi_{a,b}:K^2\to K^2$ be the associated H\'enon map.

\begin{itemize}
  	\item[{\bf (a)}] If $(a,b)\in\Hcal_\II^+$ then $J^+(\phi)$, $J^-(\phi)$, and $J(\phi)$ are nonempty, $J^+(\phi)$ is unbounded, and $J^-(\phi)=J(\phi)=\bigcap_{n\geq0}\phi^n(S_R)$; in particular $J^-(\phi)$ is bounded.
  	\item[{\bf (b)}] If $(a,b)\in\Hcal_\II^-$ then $J^+(\phi)$, $J^-(\phi)$, and $J(\phi)$ are nonempty, $J^-(\phi)$ is unbounded, and $J^+(\phi)=J(\phi)=\bigcap_{n\geq0}\phi^{-n}(S_R)$; in particular $J^+(\phi)$ is bounded.
  	\item[{\bf (c)}] If $(a,b)\in\Hcal_\III$ and $a$ is a square in $K$ then $J^+(\phi)$, $J^-(\phi)$, and $J(\phi)$ are nonempty, and $J^+(\phi)$ and  $J^-(\phi)$ are unbounded.
  	\item[{\bf (d)}] If $(a,b)\in\Hcal_\III$ and $a$ is not a square in $K$ then $J^+(\phi)$, $J^-(\phi)$, and $J(\phi)$ are empty.
\end{itemize}
\end{thm}
\begin{proof}
In this proof we use the shorthand notation $(x_n,y_n)=\phi^n(x,y)$ for the $n$-th iterate of the point $(x,y)=(x_0,y_0)$ under $\phi$, where $n\in\ZZ$.

{\bf (a)} Suppose that  $(a,b)\in\Hcal_\II^+$.  We first show that $J^+(\phi)$ is unbounded.  Proposition~\ref{Filtration2} implies in this case that 
\begin{equation}\label{Inclusion1}
\phi(S_R)\subseteq S_R
\end{equation}
and 
\begin{equation}\label{NonInclusion1}
\phi^{-1}(S_R)\not\subseteq S_R.
\end{equation}
The non-inclusion $(\ref{NonInclusion1})$ means that there exists $(x_0,y_0)\notin S_R$ for which $\phi(x_0,y_0)\in S_R$.  The inclusion $(\ref{Inclusion1})$ shows that $(x_0,y_0)$ has forward orbit contained in $S_R$ and so $(x_n,y_n)\in J^+(\phi)$ for all $n\in\ZZ$.  On the other hand, because $(x_0,y_0)$ is in $J^+(\phi)$ but not in $S_R$, Proposition~\ref{J+,J-} shows that $(x_0,y_0)\in S_R^-$.  Proposition~\ref{Filtration} part {\bf (d)} then implies that $\|(x_n,y_n)\|\to+\infty$ as $n\to-\infty$, and hence $J^+(\phi)$ is unbounded.

We next show that $J^-(\phi)=J(\phi)$.  It is trivial that $J(\phi)\subseteq J^-(\phi)$.  If $(x_0,y_0)\in J^-(\phi)$, then by Proposition~\ref{J+,J-} and Proposition~\ref{Filtration} part {\bf (f)}, some backward iterate $(x_n,y_n)$ is in $S_R$.  But then $(x_n,y_n)$ has bounded forward orbit by $(\ref{Inclusion1})$.   Consequently $(x_0,y_0)$ has bounded forward orbit; that is $(x_0,y_0)\in J^+(\phi)$ and hence $(x_0,y_0)\in J(\phi)$, completing the proof that $J^-(\phi)=J(\phi)$.

Finally, we show that $J(\phi)$ is nonempty, which trivially implies that $J^+(\phi)$ and $J^-(\phi)$ are nonempty.  In fact we have
\begin{equation}\label{JuliaIntersection}
J(\phi) = \bigcap_{n\geq0}\phi^n(S_R).
\end{equation}
For if $(x,y)$ is an element of the right hand side of $(\ref{JuliaIntersection})$, then $\phi^{-n}(x,y)\in S_R$ for all $n\geq0$, so $(x,y)\in J^-(\phi)=J(\phi)$.  Conversely, if $(x,y)\in J(\phi)$, then $\phi^{-n}(x,y)\in J(\phi)\subseteq S_R$ for all $n\geq0$ (using the $\phi$-invariance of $J(\phi)$ and Proposition~\ref{J+,J-}) and therefore $(x,y)$ is an element of the right hand side of $(\ref{JuliaIntersection})$.  The inclusion $(\ref{Inclusion1})$ implies that the right hand side of $(\ref{JuliaIntersection})$ is a nested intersection of compact sets, and so it is nonempty.

{\bf (b)}  When $(a,b)\in\Hcal_\II^-$, we have $\iota(a,b)\in\Hcal_\II^+$.  So applying part {\bf (a)} of this Proposition to $\phi_{\iota(a,b)}$ and using Lemma~\ref{FilledJuliaAndInvolution}, we obtain the statements in part {\bf (b)} of this Proposition.

{\bf (c)}  If $(a,b)\in\Hcal_\III$ and $a$ is a square in $K$, then $J^+(\phi)$, $J^-(\phi)$, and $J(\phi)$ are nonempty because, by Proposition~\ref{fp criteria}, $\phi$ has fixed points in $K^2$.  We delay the proof that $J^+(\phi)$ and $J^-(\phi)$ are unbounded until $\S$~\ref{TopConSect}.

{\bf (d)}  Assume that $(a,b)\in\Hcal_\III$.  Thus $|a|>\max(1,|b|^2)$, and in the notation of $\S$~\ref{FiltrationSect} we have $R=|a|^{1/2}$.  

We first show that if $J^+(\phi)$ is nonempty, then $a$ is a square in $K$.  Assume that $(x_0,y_0)\in J^+(\phi)$.  Then by Proposition~\ref{J+,J-} and Proposition~\ref{Filtration} part {\bf (c)}, all but finitely many points $(x_n,y_n)$ in the forward orbit of $(x_0,y_0)$ are in $S_R$, so replacing $(x_0,y_0)$ with some forward iterate, without loss of generality we may assume that $(x_n,y_n)\in S_R$ for all $n\geq0$.  In particular this implies that
\begin{equation}\label{ForwardJuliaBounds}
|x_0|=|a|^{1/2}, |x_1|\leq|a|^{1/2} \text{ and } |y_0|\leq |a|^{1/2}.
\end{equation}
The upper bounds $|x_n|\leq |a|^{1/2}$ and $|y_n|\leq |a|^{1/2}$ follow from $(x_n,y_n)\in S_R$, and if $|x_0|<|a|^{1/2}$ then we would have $|x_{1}|=|a+by_0-x_0^2|=|a|>|a|^{1/2}$, a contradiction, so $|x_0|=|a|^{1/2}$.  Finally, from $x_{1}=a+by_0-x_0^2$ we calculate $|x_0^2-a|=|by_0-x_1|<|a|=|x_0|^2$.  By Krasner's lemma and the preceding inequality, $a$ is so close to the square $x_0^2$ that $a$ itself is a square.

If $J^-(\phi)$ is nonempty, then by Lemma~\ref{FilledJuliaAndInvolution}, $J^+(\phi_{\iota(a,b)})$ is nonempty.  As $\iota(a,b)=(\frac{a}{b^2},\frac{1}{b})$, we conclude from the previous case that $\frac{a}{b^2}$ is a square in $K$, and hence $a$ is a square in $K$.  Finally, if $J(\phi)$ is nonempty then $J^\pm(\phi)$ are nonempty, and so again $a$ is a square in $K$.
\end{proof}

\begin{prop} \label{J clo com}
Let $(a,b)\in\Hcal$ and let $\phi=\phi_{a,b}:K^2\to K^2$ be the associated H\'enon map.  The sets $J^+(\phi)$, $J^-(\phi)$, and $J(\phi)$ are closed.
\end{prop}
\begin{proof}
Proposition~\ref{J+,J-} shows that the complement of $J^+(\phi)$ is a union of open sets, and hence open.  Thus $J^+(\phi)$ is closed.  Similarly, $J^-(\phi)$ is closed, and we conclude that $J(\phi) = J^+(\phi)\cap J^-(\phi)$ is closed.  
\end{proof}


\section{Regions $\Hcal_\I$ and $\Hcal_{\II}^+$: recurrence and attractors}\label{RegionsIandII+}


\subsection{Recurrence in regions $\Hcal_\I$ and $\Hcal_{\II}^+$}\label{RecurrentSect}  In this section we study the forward dynamics of the H\'enon map $\phi=\phi_{a,b}$ for $(a,b)$ in regions $\Hcal_\I$ and $\Hcal_\II^+$ of the parameter space $\Hcal$; thus we assume $|a|\leq1$ and $|b|\leq1$.   By the strong triangle inequality, it follows that $\phi$ is nonexpanding on $B_1(0,0)$; that is, for each ball $B_r(x,y)$ in $B_1(0,0)$ and each $n\geq1$, $\phi^n(B_r(x,y))\subseteq B_r(\phi^n(x,y))$.  

We say a ball $B_r(x,y)$ in $B_1(0,0)$ is $\phi$-periodic if $\phi^n(B_r(x,y))\subseteq B_r(x,y)$ for some $n\geq1$; the minimal such $n$ is the {\em minimal period} of the ball, which we denote by $m(B_r(x,y))$.  If a ball is not periodic we say it is {\em strictly preperiodic}; this is appropriate since there are only finitely many balls in $B_1(0,0)$ of any given radius, and thus all balls have finite forward orbit.  In the discussion of periodic balls it is convenient to abuse notation slightly and use $\phi^n(B_r(x,y))$ to refer to the ball $B_r(\phi^n(x,y))$; since $B_r(\phi^n(x,y))$ is the unique ball of radius $r$ containing $\phi^n(B_r(x,y))$, this causes no harm.  (Cycles of periodic balls are called {\em fuzzy cycles} by Anashin-Khrennikov \cite{MR2533085}.)

\begin{prop}\label{AttractorPeriodic}
Let $(a,b)\in\Hcal_\I\cup\Hcal_\II^+$, let $\phi=\phi_{a,b}:K^2\to K^2$ be the associated H\'enon map, and let $(x,y)\in K^2$.  Then $(x,y)\in J(\phi)$ if and only if every ball $B_r(x,y)\subseteq B_1(0,0)$ containing $(x,y)$ is $\phi$-periodic.
\end{prop}
\begin{proof}
Consider a point $(x,y)\in J(\phi)$ and a ball $B_r(x,y)$ in $B_1(0,0)$ containing it.  Since the entire backward orbit of $(x,y)$ is contained in $B_1(0,0)$, it must meet some ball $B_r(u,v)$ in $B_1(0,0)$ at least twice (as there are only finitely many balls of radius $r$ in $B_1(0,0)$).  The ball $B_r(u,v)$ is therefore periodic, and thus $B_r(x,y)$ is one of the balls in its cycle.  

Conversely, suppose $(x,y)\not\in J(\phi)$; thus $(x,y)\not\in J^-(\phi)$ by Theorem~\ref{GoodReductionCrit} and Theorem~\ref{JuliaSetBoundedNonempty}, so $\phi^{-n}(x,y)\not\in B_1(0,0)$ for some $n\geq1$.  We may therefore select $r\in|K^\times|$ so small that $\phi^{-n}(B_r(x,y))\cap B_1(0,0)=\emptyset$.  Since the ball $B_r(x,y)$ does not meet $\phi^n(B_1(0,0))$, it cannot be $\phi$-periodic.
\end{proof}

Let $\phi:M\to M$ be a continuous self-map of a metric space $M$.  A point $\alpha\in M$ is said to be {\em recurrent} for $\phi$ if, for each open neighborhood $U$ of $\alpha$, there exists some $n\geq1$ such that $\phi^n(\alpha)\in U$.  Denote by $R(\phi)$ the set of all recurrent points for $\phi$.

\begin{cor}
Let $(a,b)\in\Hcal_\I\cup\Hcal_\II^+$ and let $\phi=\phi_{a,b}:K^2\to K^2$ be the associated H\'enon map.  Then $J(\phi)=R(\phi)$.
\end{cor}
\begin{proof}
Given a point $(x,y)\in J(\phi)$, consider a ball $B_r(x,y)$ in $B_1(0,0)$ containing it.  By Proposition~\ref{AttractorPeriodic}, there exists $n\geq1$ for which $\phi^{n}(B_{r}(x,y))\subseteq B_{r}(x,y)$, and hence $(x,y)$ is recurrent.

Conversely, suppose that $(x,y)\notin J(\phi)$.  Proposition~\ref{Filtration} part {\bf (a)} implies that no points in $S_R^+$ can be recurrent, and Proposition~\ref{Filtration} part {\bf (c)} implies that no points in $S_R^-$ can be recurrent.  We are left with the case $(x,y)\in B_1(0,0)\setminus J(\phi)$.   By Proposition~\ref{AttractorPeriodic}, there exists $0<r <1$ such that $B_r(x,y)$ is not $\phi$-periodic.  Thus, $B_{r}(x,y)$ is a neighborhood of $(x,y)$ to which no forward iterate of $(x,y)$ returns, showing $(x,y)$ is not recurrent.
\end{proof}


\subsection{Attractors in region $\Hcal_{\II}^+$}
\label{StrageAttDefSect}  In this section restrict attention to maps $\phi=\phi_{a,b}$ for $(a,b)$ in region $\Hcal_\II^+$; thus we assume $|a|\leq1$ and $|b|<1$.   

Let $(M,d)$ be a metric space and let $\phi:M\to M$ be a homeomorphism.  By an {\em attractor} for $\phi$ we mean a subset $\Acal\subseteq M$ satisfying the following properties: (i) $\Acal$ is nonempty and compact; (ii) $\Acal$ is $\phi$-invariant; (iii) there exists an open set $U\subseteq M$ which properly contains $\Acal$, such that for all $\beta\in U$, 
$$
\lim_{n\to+\infty}\dist(\phi^n(\beta),\Acal)=0
$$
where $\dist(x,\Acal)=\min\{d(x,\alpha)\mid \alpha\in\Acal\}$.  The union of all such $U$ is the {\em basin of attraction} for $\Acal$.  We say the attractor $\Acal$ is {\em indecomposable} if it cannot be expressed as a disjoint union of two attractors.  (Some authors require this condition in the definition of an attractor.)  If $\Acal$ contains only one point we say it is an {\em attracting fixed point} of $\phi$.  More generally, if $\Acal$ is a $\phi$-cycle we say it is an {\em attracting cycle}.  

\begin{thm}\label{MainAttractorTheorem}
Let $(a,b)\in \Hcal_\II^+$ and let $\phi=\phi_{a,b}:K^2\to K^2$ be the associated H\'enon map.  Then $J(\phi)$ is an attractor for $\phi$, with basin of attraction containing $B_1(0,0)$.
\end{thm}
\begin{proof}
Since $|a|\leq1$ and $|b|<1$, $\phi$ restricts to a nonsurjective function $\phi:B_1(0,0)\to B_1(0,0)$, and we proved in Proposition~\ref{JuliaSetBoundedNonempty} {\bf (a)} that $J(\phi)$ is nonempty, compact, and it can be expressed as a properly nested intersection 
\begin{equation}\label{JuliaIntersection2}
J(\phi) = \bigcap_{n\geq0}\phi^n(B_1(0,0)).
\end{equation}

In order to show that $J(\phi)$ satisfies part (iii) of the definition of an attractor with $U=B_1(0,0)$ (which is open and properly contains $J(\phi)$), it suffices to show that 
$$
M_n:=\sup\{\dist((x,y),J(\phi))\mid(x,y)\in\phi^n(B_1(0,0))\}\to0 
$$
as $n\to+\infty$.  The sequence $\{M_n\}$ is nonincreasing since the intersection $(\ref{JuliaIntersection2})$ is nested.  By compactness we have  $M_n=\dist((\alpha_n,\beta_n),J(\phi))$ for some $(\alpha_n,\beta_n)\in\phi^n(B_1(0,0))$, and again using compactness and passing to a subsequence we have $(x_{n_k},y_{n_k})\to(\alpha,\beta)$ for some $(\alpha,\beta)\in B_1(0,0)$.  Since each $\phi^n(B_1(0,0))$ is closed and contains all but finitely many of the terms $\{(x_{n_k},y_{n_k})\}$, we have $(\alpha,\beta)\in\phi^n(B_1(0,0))$ for all $n\geq0$, and thus $(\alpha,\beta)\in J(\phi)$.  We conclude that $M_{n_k}\leq\|(x_{n_k},y_{n_k})-(\alpha,\beta)\|\to0$, and hence $M_n\to0$ since the sequence is nonincreasing.
\end{proof}

Attractors arising from nested intersections of the type $(\ref{JuliaIntersection2})$ are called {\em trapped attracting sets} by Milnor \cite{milnor:dynamicslectures}.

If the attractor described in Theorem~\ref{MainAttractorTheorem} is an infinite set, then it may be considered a non-Archimedean analogue of the strange attractor admitted by the real Henon map.  Based on Theorem~\ref{StrangeAttractorThmIntro} and the calculations described in $\S$~\ref{NumericalCalc}, we venture the following conjecture.

\begin{conj}\label{AttractorConjecture}
For each complete, locally compact non-Archimedean field $K$ with odd residue characteristic, there exists $(a,b)\in\Hcal_\II^+$ for which $J(\phi_{a,b})$ is an infinite set. 
\end{conj}

To investigate this conjecture further we prove Theorem~\ref{UnboundedPk}, which gives a necessary and sufficient condition for the finiteness of $J(\phi_{a,b})$ in terms of $\phi$-periodicity.  In order to prove Theorem~\ref{UnboundedPk} we require a lemma.

\begin{lem}\label{PeriodicBallsLemma}
Let $(a,b)\in \Hcal_\II^+$ and let $\phi=\phi_{a,b}:K^2\to K^2$ be the associated H\'enon map.  Let $B_r(x,y)$ be a $\phi$-periodic ball in $B_1(0,0)$ of minimal period $m$.
\begin{itemize}
	\item[{\bf (a)}]  If $B_r(x,y)\subseteq B_{r'}(x,y)\subseteq B_1(0,0)$, then $B_{r'}(x,y)$ is $\phi$-periodic with minimal period $m'\mid m$. 
	\item[{\bf (b)}]  For each $0<r'<r$ in the value group $|K^\times|$, there exists at least one $\phi$-periodic ball $B_{r'}(x',y')\subseteq B_r(x,y)$.  
	\item[{\bf (c)}]  $J(\phi)\cap B_r(x,y)$ contains a point which is either non-periodic or periodic of minimal period at least $m$.
\end{itemize}
\end{lem}
\begin{proof}
{\bf (a)}   Since $\phi^m(B_r(x,y))\subseteq  B_r(x,y)$ and $B_r(x,y)\subseteq B_{r'}(x,y)$, it follows that $\phi^m(B_{r'}(x,y))\cap B_{r'}(x,y)\neq\emptyset$.  Since $\phi$ is nonexpanding, it follows that $\phi^m(B_{r'}(x,y))\subseteq B_{r'}(x,y)$ and hence $B_{r'}(x,y)$ is periodic.  Letting $m'$ denote the minimal period of $B_{r'}(x,y)$ and letting $\ell$ denote the minimal period of $B_r(x,y)$ with respect to $\phi^{m'}$, we have $m=m'\ell$.

{\bf (b)}  Since $\phi^{m}(B_r(x,y))\subseteq B_r(x,y)$ and $\phi$ is nonexpanding, $\phi^m$ induces a self-map on the set of balls of radius $r'$ contained in $B_r(x,y)$; since this set is finite, there is a periodic ball.

{\bf (c)}  Iterating part {\bf (b)}, there exists a nested sequence of $\phi$-periodic balls
$$
B_r(x,y)\supseteq B_{r_1}(x_1,y_1)\supseteq B_{r_2}(x_2,y_2)\supseteq \dots
$$
where $r_k\to0$.  By compactness,
$$
\bigcap_{k\geq1}B_{r_k}(x_k,y_k)=\{(x_0,y_0)\}
$$
for some $(x_0,y_0)\in B_r(x,y)$.  By part {\bf (a)}, all balls in $B_1(0,0)$ containing $(x_0,y_0)$ are $\phi$-periodic, and so $(x_0,y_0)\in J(\phi)$ by Proposition~\ref{AttractorPeriodic}.  If $(x_0,y_0)$ is not $\phi$-periodic then there is nothing left to prove.  If $(x_0,y_0)$ is $\phi$-periodic with minimal period $m_0$, then $\phi^{m_0}(B_r(x,y))\cap B_r(x,y)\neq\emptyset$ as both $\phi^{m_0}(B_r(x,y))$ and $B_r(x,y)$ contain $(x_0,y_0)$.  Since $\phi$ is nonexpanding, it follows that $\phi^{m_0}(B_{r}(x,y))\subseteq B_{r}(x,y)$ and hence the minimal period $m$ of $B_{r}(x,y)$ is at most $m_0$.
\end{proof}

\begin{thm}\label{UnboundedPk}
Let $(a,b)\in \Hcal_\II^+$ and let $\phi=\phi_{a,b}:K^2\to K^2$ be the associated H\'enon map.  Then $J(\phi)$ is finite if and only if there exists some $N\geq1$ such that $m(B_r(x,y))\leq N$ for all $\phi$-periodic balls $B_r(x,y)$ in $B_1(0,0)$.
\end{thm}
\begin{proof}
If $J(\phi)$ is finite then it contains only periodic points; let $N$ be the largest minimal period among all of these points.  If $B_r(x,y)$ is a $\phi$-periodic ball, then it intersects nontrivially with $J(\phi)$ by Lemma~\ref{PeriodicBallsLemma}, and so it must contain a periodic point $(\alpha,\beta)$.  The minimal period of $B_r(x,y)$ cannot be greater than the minimal period of $(\alpha,\beta)$, and hence $m(B_r(x,y))\leq N$.

Conversely, assume that $J(\phi)$ is an infinite set.  Let $n\geq1$ be arbitrary, and let $(x,y)\in J(\phi)$ be a point which is either not periodic, or else periodic with minimal period at least $n$.  Such a point must exist because, by a Bezout theorem argument (e.g. \cite{MR991490}), $\phi$ has only finitely many periodic points of any given minimal period.  Since the points $(x,y),\phi(x,y),\dots,\phi^{m-1}(x,y)$ are distinct, there exists $r\in|K^\times|$ so small that the balls 
\begin{equation}\label{DisjointBalls}
B_{r}(x,y),\phi(B_{r}(x,y)),\dots,\phi^{n-1}(B_{r}(x,y))
\end{equation} 
are disjoint.  We know that $B_{r}(x,y)$ is $\phi$-periodic by Proposition~\ref{AttractorPeriodic} and the fact that $(x,y)\in J(\phi)$, and the minimal $\phi$-period of $B_{r}(x,y)$ must be at least $n$, because the balls $(\ref{DisjointBalls})$ are disjoint.  As $n$ was arbitrary, it follows that the $N$ described in the statement of the theorem does not exist.
\end{proof}

Recalling our assumptions that $|a|\leq 1$ and $|b|<1$, we now prove a result which further specializes to the case $|a|<1$.  In this case the attractor $J(\phi_{a,b})$ is the union of two attractors, one of which is an attracting fixed point.  

\begin{prop}
Let $(a,b)\in \Hcal_\II^+$ and assume further that $|a|<1$.  Let $\phi=\phi_{a,b}:K^2\to K^2$ be the associated H\'enon map. 
\begin{itemize}
	\item[{\bf (a)}]  $\phi$ has an attracting fixed point $(c,c)\in B_1^\circ(0,0)$.
	\item[{\bf (b)}]  $\Acal=J(\phi)\setminus\{(c,c)\}$ is an attractor for $\phi$.  
	\item[{\bf (c)}]  $B_1(0,0)$ is the smallest polydisc in $K^2$ containing $J(\phi)$. 
\end{itemize}
\end{prop}
\begin{proof}
{\bf (a)}  Inspection of the Newton polygon of the fixed point equation $x^2-(b-1)x-a=0$ shows that it has one root $c\in K$ with $|c|=|a|<1$, and another root $d$ with $|d|=1$, and thus $(c,c)$  and $(d,d)$  are fixed points.  To verify that $\{(c,c)\}$ is attracting, we partition $B_1(0,0)$ into the two sets
\begin{equation}\label{AttFixedPartition}
\begin{split}
U & = \{(x,y)\in B_1(0,0) \mid |x|<1\} \\
V & = \{(x,y)\in B_1(0,0) \mid |x|=1\}
\end{split}
\end{equation}
and we will show that $U$ is a basin of attraction for $(c,c)$.  Conjugating by $(x,y)\mapsto (x+c,y+c)$ (which preserves the sets $U$ and $V$), it suffices to show that $(0,0)$ is attracting for the map
$$
\psi(x,y)=\phi(x+c,y+c)-(c,c)=(by-2cx-x^2,x).
$$
If $(x,y)\in U$ with $|x|\leq r<1$, set $X=by-2cx-x^2$ and $Y=x$.  Then $|X|\leq \max(|b|r,|c|r,r^2)$ and $|Y|\leq r$, so
\begin{equation}\label{psiBound}
\|\psi^2(x,y)\|=\|\psi(X,Y)\|\leq\max(|b|r,|c|r,r^2).
\end{equation}
Since $|b|<1$ and $|c|=|a|<1$, we conclude $\psi^n(x,y)\to(0,0)$ as $n\to+\infty$.

{\bf (b)}  It is easy to see using the strong triangle inequality that $\phi(V)\subseteq V$.  We now claim that $J(\phi)\cap U=\{(c,c)\}$; in other words, no point of $U$ is in $J(\phi)$ except the attracting fixed point $(c,c)$ itself.  For if there exists some $(x,y)\in J(\phi)\cap U$ with $(x,y)\neq(c,c)$, then the entire backward orbit $\{\phi^{-n}(x,y)\mid n\geq0\}$ must be contained in $U$ by Proposition~\ref{J+,J-} and the fact that $\phi(V)\subseteq V$.  Again conjugating as in part {\bf (a)}, it follows that $\psi^{-n}(x-c,y-c)\in U$ for all $n\geq1$.  But then $\|\psi^{-2n}(x-c,y-c)\|$ would be strictly increasing as $n\to+\infty$ by $(\ref{psiBound})$, an impossibility as $U$ is bounded and $K$ is discretely valued.

Set $\Acal=J(\phi)\setminus\{(c,c)\}$, which is nonempty because it contains the fixed point $(d,d)$.  Since $\Acal\subseteq V$, it now follows from $(\ref{JuliaIntersection2})$ that $\Acal = \bigcap_{n\geq0}\phi^n(V)$.  The proof that $\Acal$ is an attractor now follows from the same argument used to prove Theorem~\ref{MainAttractorTheorem}.

{\bf (c)}  Suppose $D$ is a polydisc such that $J(\phi) \subset D$. Then $(c,c)\in D$ and so we may write $D=D_{r_1,r_2}(c,c)$ for some radii $r_1,r_2$.  But $(d,d)\in D$ as well and $|c-d|=|d|=1$ and so $r_1\geq1$ and $r_2\geq1$.
\end{proof}

\subsection{Examples in $\QQ_3$}\label{StrangeAttSection}  In this section we explore two examples over the field $\QQ_3$ of $3$-adic numbers.   As usual, denote by $|\cdot|_3$ the absolute value on $\QQ_3$, normalized so that $|3|_3=1/3$, and set $\ZZ_3=\{x\in\QQ_3\mid|x|_3\leq1\}$.

\begin{thm}\label{StrangeAttractorThm}
For $a\in D_{1/9}(2)$, define $\phi=\phi_{a,3}:\QQ_3^2\to\QQ_3^2$ by $\phi(x,y)=(a+3y-x^2,x)$.
\begin{itemize}
	\item[{\bf (a)}]  $B_{1/3}(1,1)$ is fixed by $\phi$ and the other eight balls of radius $1/3$ in $B_{1}(0,0)$ are strictly preperiodic.  For each $k\geq1$, there is a cycle of balls of radius $1/3^{k+1}$ in $B_{1/3}(1,1)$ of minimal period $3^k$, and all other balls of radius $1/3^{k+1}$ in $B_{1/3}(1,1)$ are strictly preperiodic.  Moreover, each periodic ball of radius $1/3^{k}$ contains exactly three periodic balls of radius $1/3^{k+1}$.
	\item[{\bf (b)}]  The attractor $J(\phi)$ is uncountably infinite, has Haar measure zero in $\QQ_3^2$, and contains no periodic points.  Each point of $J(\phi)$ has dense forward orbit in $J(\phi)$.  In particular, $J(\phi)$ is indecomposable. 
	\item[{\bf (c)}]  There exists a probability measure $\mu_\phi$ supported on $J(\phi)$ with the property that the forward orbit of any point in $B_1(0,0)$ is $\mu_\phi$-equidistributed; in other words, 
$$
\lim_{n\to+\infty}\frac{1}{n}\sum_{i=1}^{n}f(\phi^i(x,y))=\int fd\mu_\phi
$$
for all $(x,y)\in B_1(0,0)$ and all continuous $f:B_1(0,0)\to\RR$.
\end{itemize}
\end{thm}

Since the automorphism $h(x,y)=(x+1,x+y+1)$ fixes $B_1(0,0)$ and permutes the balls of any given radius in $B_1(0,0)$, we may replace $\phi$ with its conjugate 
$$
\psi(x,y) = h^{-1}\circ\phi\circ h(x,y)=(a+1-2x+3y-x^2,-a-1+3x-3y+x^2)
$$
and replace the ball $B_{1/3}(1,1)$ which is fixed by $\phi$, with $B_{1/3}(0,0)$ which is fixed by $\psi$.  

\begin{lem}\label{psikLemma}
For each $k\geq1$, there exists $c_k\in \ZZ_3$ and a polynomial map $\psi_k:B_1(0,0)\to B_1(0,0)$ which has coefficients in $\ZZ_3$ and has the following properties.
\begin{itemize}
	\item $\psi_1=\psi$ and $c_1=-1$.
	\item $\psi_{k+1}=h_k^{-1}\circ\psi_k^3\circ h_k$ for all $k\geq1$, where $h_{k}(x,y)=(3x,3y+3c_k)$.  
	\item $\psi_k$ fixes $B_{1/3}(0,0)$, has a $3$-cycle $\{B_{1/9}(0,-3c_k),B_{1/9}(3,-3c_k),B_{1/9}(6,-3c_k)\}$ of balls of radius $1/9$, and all other balls of radius $1/9$ in $B_{1/3}(0,0)$ are strictly preperiodic to this cycle.
\end{itemize}
\end{lem}
\begin{proof}
It will be useful to define the ideal $I=(9,3x,3y,x^2,xy,y^2)$ of $\ZZ_3[x,y]$.  Thus a polynomial $A +B x+C y+\dots\in\ZZ_3[x,y]$ is an element of $I$ if and only if $9\mid A$, $3\mid B$, and $3\mid C$ in $\ZZ_3$.  Given a polynomial map $f:B_{1/3}(0,0)\to B_{1/3}(0,0)$ having coefficients in $\ZZ_p$, the action of $f$ on the nine balls of radius $1/9$ contained in $B_{1/3}(0,0)$ depends only on the congruence class of $f$ modulo $I$; this observation will simplify the calculations.

In addition to the properties in the statement of the Lemma, we will also show that
\begin{equation}\label{PsikCongruence}
\psi_k(x,y)\equiv (3+x,3c_k) \pmod{I}
\end{equation}
for all $k\geq1$.  With $\psi_1=\psi$ and $c_1=-1$, then using the assumption that $a\equiv2\pmod{9}$, we have $\psi_1(x,y)\equiv(3+x,-3)\pmod{I}$, which is $(\ref{PsikCongruence})$, and from which it follows that $\psi_1$ has a $3$-cycle $\{B_{1/9}(0,6)$, $B_{1/9}(3,6)$, $B_{1/9}(6,6)\}$, and that all of the other six balls of radius $\frac{1}{9}$ in $B_{1/3}(0,0)$ are strictly preperiodic into this cycle.  

Assuming $\psi_k$ and $c_k$ have already been constructed as in the statement of the Lemma and satisfying $(\ref{PsikCongruence})$, set $\psi_{k+1}=h_k^{-1}\circ\psi_k^3\circ h_k$, 
where 
\begin{equation}\label{hkDef}
h_{k}(x,y)=(3x,3y+3c_k).
\end{equation}
Using $(\ref{PsikCongruence})$ and the definition of $I$ we may write 
$$
\psi_k(x,y)=(3+x+9a_k+F(x,y),3c_k+9b_k+G(x,y))
$$ 
where $a_k,b_k\in\ZZ_3$, and both $F$ and $G$ are in $I$ and have vanishing constant term.  We calculate
$$
\textstyle h_k^{-1}\circ\psi_k\circ h_k(x,y)=(1+3a_k+x+\frac{1}{3}F(3x,3y+3c_k),3b_k+\frac{1}{3}G(3x,3y+3c_k)).
$$
Since both $F$ and $G$ have vanishing constant term and are in $I$, it follows that both $\frac{1}{3}F(3x,3y+3c_k)$ and $\frac{1}{3}G(3x,3y+3c_k)$ are in $I$, and we deduce
$$
\textstyle h_k^{-1}\circ\psi_k\circ h_k(x,y)\equiv(1+3a_k+x,3b_k)\pmod{I}.
$$
Finally, iterating we arrive at
$$
\psi_{k+1}(x,y)=h_k^{-1}\circ\psi_k^3\circ h_k(x,y)\equiv(3+x,3b_k)\pmod{I},
$$
and thus we set $c_{k+1}=b_k$, establishing the desired congruence $(\ref{PsikCongruence})$.  The final sentence of the Lemma follows easily from the congruence $(\ref{PsikCongruence})$.
\end{proof}

\begin{proof}[Proof of Theorem~\ref{StrangeAttractorThm} {\bf (a)}]  As explained above, it suffices to prove the analogous statement for $\psi$ instead of $\phi$.  Using $\psi(x,y)\equiv(-2x-x^2,x^2)\pmod{3}$, it is elementary to check that $B_{1/3}(0,0)$ is fixed by $\psi$ and all of the other eight balls of radius $\frac{1}{3}$ in $B_1(0,0)$ are strictly preperiodic into this fixed ball.

We proceed by induction on $k$.  The $k=1$ case follows from the Lemma~\ref{psikLemma} and the fact that $\psi_1=\psi$.  Fix $k\geq1$, and assume that $\psi$ has a cycle of balls of radius $1/3^{k+1}$ in $B_{1/3}(0,0)$ of minimal period $3^k$, that all other balls of radius $1/3^{k+1}$ in $B_{1/3}(0,0)$ are strictly preperiodic into this cycle, and that each periodic ball of radius $1/3^{k}$ contains exactly three periodic balls of radius $1/3^{k+1}$.  

Note that $\psi_{k+1}=H_k^{-1}\circ\psi^{3^{k}}\circ H_k$ where $H_k=h_1\circ h_2\circ\dots \circ h_k$.  By $(\ref{hkDef})$, the automorphism $H_k$ takes discs of radius $1/3^r$ to discs of radius $1/3^{r+k}$.  It then follows from the final sentence of Lemma~\ref{psikLemma} that $\psi^{3^{k}}$ has a $3$-cycle 
\begin{equation}\label{AttractorThreeCycle}
\{B_{1/3^{k+2}}(x_0,y_0),\psi^{3^{k}}(B_{1/3^{k+2}}(x_0,y_0)),\psi^{3^{k}2}(B_{1/3^{k+2}}(x_0,y_0))\}
\end{equation}
of balls of radius $1/3^{k+2}$, all of which are contained in the ball $B_{1/3^{k+1}}(x_0,y_0)$ which is fixed by $\psi^{3^k}$, and that $B_{1/3^{k+1}}(x_0,y_0)$ contains no periodic balls of radius $1/3^{k+2}$ except the three balls in $(\ref{AttractorThreeCycle})$.  In particular, 
\begin{equation}\label{AttractorCycle1}
\{B_{1/3^{k+2}}(x_0,y_0),\psi(B_{1/3^{k+2}}(x_0,y_0)),\dots,\psi^{3^{k+1}-1}(B_{1/3^{k+2}}(x_0,y_0))\}
\end{equation}
is a $3^{k+1}$-cycle for $\psi$.  By the induction hypothesis, it follows that
\begin{equation}\label{AttractorCycle2}
\{B_{1/3^{k+1}}(x_0,y_0),\psi(B_{1/3^{k+1}}(x_0,y_0)),\dots,\psi^{3^{k}-1}(B_{1/3^{k+1}}(x_0,y_0))\}
\end{equation}
is a $3^k$-cycle for $\psi$, and the only periodic balls of radius $1/3^{k+1}$ are in the cycle $(\ref{AttractorCycle2})$.

Suppose that $B_{1/3^{k+2}}(x_0',y_0')$ is a periodic ball of radius $1/3^{k+2}$ not occurring in the cycle $(\ref{AttractorCycle1})$; then the cycle containing $B_{1/3^{k+2}}(x_0',y_0')$ is disjoint from the cycle $(\ref{AttractorCycle1})$.  Since the larger ball $B_{1/3^{k+1}}(x_0',y_0')$ is periodic, it is equal to one of the balls in $(\ref{AttractorCycle2})$, whereby some ball in the cycle of $B_{1/3^{k+2}}(x_0',y_0')$ would be contained in the ball $B_{1/3^{k+1}}(x_0,y_0)$ but would not be equal to one of the three balls in $(\ref{AttractorThreeCycle})$, a contradiction.  So $(\ref{AttractorCycle1})$ are the only periodic balls of radius $1/3^{k+2}$.

Finally, we observe that each ball $\psi^{r}(B_{1/3^{k+1}}(x_0,y_0))$ in $(\ref{AttractorCycle2})$ contains the three balls $\psi^{r}(B_{1/3^{k+2}}(x_0,y_0))$, $\psi^{r+3^k}(B_{1/3^{k+2}}(x_0,y_0))$, and $\psi^{r+3^k2}(B_{1/3^{k+2}}(x_0,y_0))$ from the cycle $(\ref{AttractorCycle1})$.
\end{proof}

\begin{proof}[Proof of Theorem~\ref{StrangeAttractorThm} {\bf (b)}]  For each $k\geq1$ define $\Per_k$ to be the set of $\phi$-periodic balls in $B_1(0,0)$ of radius $1/3^k$.  By Theorem~\ref{AttractorPeriodic}, 
$$
J(\phi) = \bigcap_{k\geq1}\bigcup_{\Bcal\in\Per_k}\Bcal.
$$
Each ball in $\Per_k$ contains exactly three balls in $\Per_{k+1}$, so arbitrarily indexing each such triple of balls using the set $\{1,2,3\}$, we see that $J(\phi)$ is in bijective correspondence with $\{1,2,3\}^\NN$, and hence is uncountable.  Since $|\Per_k|=3^{k-1}$ and a ball of radius $1/3^k$ in $\QQ_3^2$ has Haar measure $1/3^{2k}$, we see that for each $k\geq1$ the Haar measure of $J(\phi)$ is at most $(3^{k-1})(1/3^{2k})=3^{-k-1}\to0$ as $k\to+\infty$, and therefore the Haar measure of $J(\phi)$ is zero.

If $J(\phi)$ contains a periodic point $(x_0,y_0)$, then Lemma~\ref{PeriodicBallsLemma} implies that its minimal period is at least the minimal period of $B_{1/3^k}(x_0,y_0)$ for all $k\geq1$.  But every ball in $\Per_k$ has minimal period $3^{k-1}$, a contradiction as $k\to+\infty$.

If $(x_0,y_0)$ and $(x_0',y_0')$ are two points in $J(\phi)$ and $B_{1/3^k}(x_0',y_0')$ is a neighborhood of $(x_0',y_0')$, then some point in the forward orbit  orbit of $(x_0,y_0)$ lies in $B_{1/3^k}(x_0',y_0')$, because both $B_{1/3^k}(x_0,y_0)$ and $B_{1/3^k}(x_0',y_0')$ are balls in the same $3^{k-1}$-cycle.  This shows that the forward orbit of $(x_0,y_0)$ is dense in $J(\phi)$.  In particular this implies the indecomposability of $J(\phi)$, for if $J(\phi)=\Acal_1\cup\Acal_2$ for disjoint attractors $\Acal_1$ and $\Acal_2$, then as $\Acal_1$ is $\phi$-invariant, the forward orbit of any point of $\Acal_1$ cannot be dense in $J(\phi)$.
\end{proof}

\begin{proof}[Proof of Theorem~\ref{StrangeAttractorThm} {\bf (c)}]  We now construct the measure $\mu_\phi$.  For each nonempty finite subset $X$ of $B_1(0,0)$, denote by $[X]=\frac{1}{|X|}\sum_{(x,y)\in X}\delta_{(x,y)}$, the probability measure supported equally on each point of $X$.  Here $\delta_{(x,y)}$ denotes the Dirac measure supported at a point $(x,y)\in B_1(0,0)$.  

For each $k\geq1$, let $X_k$ be a subset of $B_1(0,0)$ consisting of precisely one point from each $\phi$-periodic ball $\Bcal\in\Per_k$.  Thus $|X_k|=3^{k-1}$ by part {\bf (a)}.  Given a continuous function $f:B_1(0,0)\to\RR$, we will show that the limit 
\begin{equation}\label{DistPerLimit}
L_f=\lim_{k\to+\infty}\int f d[X_k]
\end{equation}
exists.  Fix $\epsilon>0$, and let $k\geq1$ be an integer so large that $|f(x',y')-f(x,y)|\leq\epsilon$ whenever $\|(x',y')-(x,y)\|\leq1/3^k$;  such uniform continuity is guaranteed by the compactness of $B_1(0,0)$.  If $k'\geq k$, then by part {\bf (a)}, $X_{k'}$ consists of $3^{k'-1}$ points, precisely $3^{k'-k}$ of which occur in each ball in $\Per_k$.  Of course, $X_k$ contains precisely one point in each ball $\Bcal$ in $\Per_k$; call this point $(x_\Bcal,y_\Bcal)$.  We then have
\begin{equation}\label{DistPerEstimate}
\begin{split}
\bigg|\int f d[X_{k'}]-\int f d[X_k]\bigg| & = \bigg|\frac{1}{3^{k'-1}}\bigg(\sum_{(x',y')\in X_{k'}}f(x',y')-3^{k'-k}\sum_{(x,y)\in X_{k}}f(x,y)\bigg)\bigg| \\
	& \leq \frac{1}{3^{k'-1}}\sum_{\Bcal\in\Per_k}\sum_{(x',y')\in X_{k'}\cap\Bcal}|f(x',y')-f(x_\Bcal,y_\Bcal)| \\
	& \leq\epsilon.
\end{split}
\end{equation}
We conclude that the sequence $\{\int f d[X_k]\}$ is Cauchy and hence the limit $(\ref{DistPerLimit})$ exists.

By Prokhorov's theorem (\cite{MR1700749} Thm 5.1), the sequence $\{[X_k]\}$ of measures has a subsequence $\{[X_{k_\ell}]\}$ converging weakly to some probability measure $\mu_\phi$.  Since the limit $(\ref{DistPerLimit})$ exists for each continuous $f:B_1(0,0)\to\RR$, we must therefore have $L_f=\int fd\mu_\phi$, showing that in fact $[X_k]\to\mu_\phi$ weakly.

Finally, we show that the forward orbit of any point $(x_0,y_0)\in B_1(0,0)$ is equidistributed with respect to $\mu_\phi$.  In other words, for each integer $n\geq1$, let 
$$
Y_n=\{\phi(x_0,y_0), \phi^2(x_0,y_0), \dots, \phi^n(x_0,y_0)\}
$$
be the first $n$ points in the forward orbit of $(x_0,y_0)$.  We will show that $[Y_n]\to\mu_\phi$ weakly as $n\to+\infty$.

We first remark that $(x_0,y_0)$ is not periodic, because we have shown that the attractor $J(\phi)$ contains no periodic points, and that $B_1(0,0)$ is a basin of attraction for this attractor.    Fix $\epsilon>0$ and a continuous function $f:B_1(0,0)\to\RR$.  Let $k$ be an integer so large that $|f(x',y')-f(x,y)|\leq\epsilon$ whenever $\|(x',y')-(x,y)\|\leq1/3^k$.  Since there are only finitely many balls in $B_1(0,0)$ of radius $1/3^k$, there exists an integer $n_0\geq1$ for which $\Bcal=B_{1/3^k}(\phi^{n_0}(x_0,y_0))$ is a $\phi$-periodic ball (hence with minimal $\phi$-period $3^{k-1}$).  

Set $m=3^{k-1}$.  For each integer $n\geq n_0$, we may partition the partial orbit $Y_n$ as
\begin{equation}\label{DistPerPartition}
Y_n = T\cup C_1\cup C_2\cup \dots \cup C_r\cup C_{r+1}^*,
\end{equation}
where
$$
T=\{\phi(x_0,y_0), \phi^2(x_0,y_0), \dots, \phi^{n_0-1}(x_0,y_0)\},
$$
each set $C_i$ is a segment of the partial orbit $Y_n$ consisting of precisely one point from each ball in the $\phi$-cycle $\{\Bcal, \phi(\Bcal), \dots, \phi^{m-1}(\Bcal)\}$, and $C_{r+1}^*$ consists of the final $m^*$ points of $Y_n$, where $0\leq m^*< m$.  Thus $C_{r+1}^*$ is either the empty set, or it consists of precisely one point from each ball in the incomplete cycle $\{\Bcal, \phi(\Bcal), \dots, \phi^{m^*-1}(\Bcal)\}$.  

Each set $C_i$ (for $1\leq i\leq r$) may be taken as the set $X_k$ in $(\ref{DistPerEstimate})$, and taking $k'\to+\infty$ in that estimate and using the weak convergence $[X_{k'}]\to\mu_\phi$, we have deduced $|\int f d[C_i]-L_f|\leq\epsilon$.  By the partition $(\ref{DistPerPartition})$ and the definition of the probability measures $[X]$, we have
\begin{equation}\label{DistPerPartition2}
[Y_n] = \frac{n_0-1}{n}[T]+ \sum_{i=1}^{r}\frac{m}{n}[C_i]+\frac{m^*}{n}[C_{r+1}^*].
\end{equation}
Therefore 
\begin{equation*}\label{DistPerEstimate2}
\int fd[Y_n]-L_f = \int (f-L_f)d[Y_n] =I_1+I_2
\end{equation*}
where 
\begin{equation*}\label{DistPerEstimate3}
I_1=\frac{m}{n}\sum_{i=1}^{r}\bigg(\int fd[C_i]-L_f\bigg)
\end{equation*}
and 
\begin{equation*}\label{DistPerEstimate4}
I_2=\frac{n_0-1}{n}\int(f-L_f)d[T] + \frac{m^*}{n}\int(f-L_f)d[C_{r+1}^*].
\end{equation*}
Since the sets $C_1,\dots, C_r$ each contain $m$ points, we have $rm\leq n$, and therefore $|I_1|\leq\frac{rm}{n}\epsilon\leq\epsilon$.  Since $f$ is bounded, and since both $n_0$ and $m^*<m=3^{k-1}$ are bounded independently of $n$, taking $n\to+\infty$ we deduce
$$
\limsup_{n\to+\infty}\bigg|\int fd[Y_n]-L_f\bigg|\leq\epsilon.
$$
As $\epsilon$ was arbitrary, we conclude $\lim_{n\to+\infty}|\int fd[Y_n]-L_f|=0$, and as $f$ was arbitrary and $L_f=\int fd\mu_\phi$, we have established the weak convergence $[Y_n]\to\mu_\phi$.
\end{proof}

\begin{thm}\label{Attracting2CycleThm}
For $a\in D_{1/3}(1)$, define $\phi=\phi_{a,3}:\QQ_3^2\to\QQ_3^2$ by $\phi(x,y)=(a+3y-x^2,x)$.  Then $J(\phi)$ is an attracting $2$-cycle.
\end{thm}

\begin{proof}
Since $|a|_3=1$ and $|3|_3<1$, $(a,3)$ is in region $\Hcal_\II^+$, and so $J(\phi)$ is an attractor.  It follows from Proposition~\ref{fp criteria} that $\phi$ has a $2$-cycle $\{(c,d),(d,c)\}$, where $c \in D_{1/3}(0)$ and $d  \in D_{1/3}(1)$ are the two roots of $x^2+2x+(4-a)=0$.  These roots exist by Hensel's lemma, as $x^2+2x+(4-a)\equiv x(x-1)\pmod{3}$.  Our goal is to show that $J(\phi)=\{(c,d),(d,c)\}$.  In other words, $J(\phi)$ is an attracting $2$-cycle.

By elementary calculations modulo $3$, each of the nine balls of radius $1/3$ in $B_1(0,0)$ is mapped into one of the two balls $B_{1/3}(0,1)=B_{1/3}(c,d)$ or $B_{1/3}(1,0)=B_{1/3}(d,c)$ after two iterations of $\phi$.  In other words
\begin{equation}\label{GetRidOfTails}
\phi^{2}(B_1(0,0))\subseteq B_{1/3}(c,d)\cup B_{1/3}(d,c).
\end{equation}
Since $(c,d)$ is a fixed point of $\phi^2$, we may write 
$$
\psi(x,y):=\phi^2(x+c,y+d)-(c,d)=(A_1x+B_1y+F_1(x,y),A_2x+B_2y+F_2(x,y)),
$$
where the $F_i\in\ZZ_3[x,y]$ have vanishing constant and linear part.  A straightforward calculation shows that $|A_i|_3\leq1/3$ and $|B_i|_3\leq1/3$, from which it follows that $\|\psi(x,y)\|\leq\frac{1}{3}\|(x,y)\|$ for $(x,y)\in B_{1/3}(0,0)$.  Hence
\begin{equation}\label{Attracting2Cycle1}
\textstyle\|\phi^2(x,y)-(c,d)\| \leq\frac{1}{3} \|(x,y)-(c,d)\| \text{ whenever } (x,y)\in B_{1/3}(c,d),
\end{equation}
and iterating this last inequality gives
\begin{equation}\label{Attracting2Cycle2}
\phi^{2n}(B_{1/3}(c,d))\subseteq B_{1/3^{n+1}}(c,d)
\end{equation}
for all $n\geq0$.  Using $(\ref{Attracting2Cycle2})$ we also have, for $n\geq1$,
\begin{equation}\label{Attracting2Cycle5}
\begin{split}
\phi^{2n}(B_{1/3}(d,c)) & = \phi^{2n-1}(\phi(B_{1/3}(d,c))) \\
	& \subseteq \phi^{2n-1}(B_{1/3}(c,d)) \\
	& = \phi(\phi^{2(n-1)}(B_{1/3}(c,d))) \\
	& \subseteq \phi(B_{1/3^n}(c,d)) \\	
	& \subseteq B_{1/3^n}(d,c).	
\end{split}
\end{equation}
using $\phi(c,d)=(d,c)$ and that $\phi$ is nonexpanding.  It now follows from $(\ref{JuliaIntersection2})$, $(\ref{GetRidOfTails})$, $(\ref{Attracting2Cycle2})$, and $(\ref{Attracting2Cycle5})$ that
\begin{equation*}
\begin{split}
J(\phi) & \subseteq \bigcap_{n\geq1} \phi^{2n+2}(B_{1}(0,0)) \\
	& \subseteq \bigcap_{n\geq1} \phi^{2n}(B_{1/3}(c,d)\cup B_{1/3}(d,c)) \\
	& = \bigg(\bigcap_{n\geq1} \phi^{2n}(B_{1/3}(c,d))\bigg)\cup\bigg(\bigcap_{n\geq1} \phi^{2n}(B_{1/3}(d,c))\bigg) \\
	& \subseteq \bigg(\bigcap_{n\geq1} B_{1/3^{n+1}}(c,d)\bigg)\cup\bigg(\bigcap_{n\geq1} B_{1/3^{n}}(d,c)\bigg) \\
	& = \{(c,d),(d,c)\}.
\end{split}
\end{equation*}
Thus $J(\phi)=\{(c,d),(d,c)\}$, since periodic points are always elements of $J(\phi)$.
\end{proof}

\subsection{Further speculation on the attractor $J(\phi_{a,b})$}\label{NumericalCalc}  We now record some numerical calculations which are suggestive of further examples similar to Theorem~\ref{StrangeAttractorThm} and Theorem~\ref{Attracting2CycleThm}.  In the following table, $p$ is an odd prime, $(a,b)$ is a point in region $\Hcal_\II^+$ of the parameter space over $\QQ_p$, $\phi_{a,b}:\QQ_p^2\to\QQ_p^2$ is the corresponding H\'enon map, and $P_k$ denotes the largest minimal period among all balls of radius $1/p^k$ in $B_1(0,0)$.  Recall from Theorem~\ref{UnboundedPk} that the attractor $J(\phi_{a,b})$ is an infinite set if and only if the sequence $\{P_k\}$ is unbounded.  An entry $m^*$ indicates that we have only verified that $P_k\geq m$.

\begin{table}[h!]
\centering
\begin{tabular}{cc|cccccc}
$p$ & $\phi_{a,b}$ & $P_1$ & $P_2$ & $P_3$ & $P_4$ & $P_5$ & $P_{6}$  \\
\hline
$3$ & $\phi_{2,3}$ & $1$ & $3$ & $9$ & $27$ & $81$ & $243$  \\
$3$ & $\phi_{8,3}$ & $1$ & $1$ & $3$ & $9$ & $27$ & $81$  \\
$3$ & $\phi_{2,9}$ & $1$ & $1$ & $3$ & $9$ & $27$ & $81$  \\
$5$ & $\phi_{4,5}$ & $3$ & $3$ & $3$ & $3$ & $3$ & $3$ \\
$5$ & $\phi_{1,5}$ & $2$ & $5$ & $25$ & $125$ & $625^*$ & $3125^*$  \\
$7$ & $\phi_{1,7}$ & $2$ & $2$ & $2$ & $2$ & $2$ & $2$  \\
$7$ & $\phi_{2,7}$ & $1$ & $6$ & $42$ & $294^*$ & $2058^*$ & $14406^*$  \\
\end{tabular} 
\bigskip
\caption{The maximal cycle lengths $P_k$ of balls of radius $1/p^k$, for $1\leq k\leq 6$ and various examples of $p$-adic H\'enon maps.}
\end{table}

For reference, the first line of this table refers to a map $\phi_{2,3}$ which is included in Theorem~\ref{StrangeAttractorThm}.  Again let $\Per_k$ denote the set of periodic balls of radius $1/p^k$ in $B_1(0,0)$.  There are some notable differences in the cycle structures of $\Per_k$ for the three $3$-adic maps occurring in this table.  As we know from Theorem~\ref{StrangeAttractorThm}, for $\phi_{2,3}$, $\Per_k$ is a single $3^{k-1}$-cycle for all $k\geq1$.  For the map $\phi_{8,3}$, our calculations show that $\Per_1$ contains one fixed point; $\Per_2$ contains three fixed points; $\Per_3$ contains one $3$-cycle and six fixed points; $\Per_4$ contains one $9$-cycle, four $3$-cycles, and six fixed points; and $\Per_5$ contains one $27$-cycle, four $9$-cycles, four $3$-cycles, and six fixed points.  For the map $\phi_{2,9}$, our calculations show that $\Per_1$ contains one fixed point; $\Per_2$ contains three fixed points; $\Per_3$ contains three $3$-cycles; $\Per_4$ contains three $9$-cycles; and $\Per_5$ contains three $27$-cycles.

Over $\QQ_5$, the map $\phi_{4,5}:\QQ_5^2\to\QQ_5^2$ has the property that $\Per_k$ is a single $3$-cycle for all $1\leq k\leq 6$.  Thus $J(\phi_{4,5})$ is likely an attracting $3$-cycle.  The map $\phi_{1,5}:\QQ_5^2\to\QQ_5^2$ has the property that $\Per_k$ contains a single $2$-cycle and a single $5^{k-1}$-cycle for all $1\leq k\leq 4$.  If this pattern continues, then $J(\phi_{1,5})$ is the union of an attracting $2$-cycle and an indecomposable attractor similar to the $3$-adic attractor described in Theorem~\ref{StrangeAttractorThm}.

Over $\QQ_7$, the table suggests that $J(\phi_{1,7})$ is an attracting $2$-cycle.  It would be straightforward to give a proof of this with an adaptation of the proof of Theorem~\ref{Attracting2CycleThm}.  The example $\phi_{2,7}$ is notable in that $\Per_k$ contains large cycle lengths that are not powers of $7$; it appears that all cycles lengths in $\Per_k$ are of the form $1$, $3\cdot 7^\ell$, or $6\cdot 7^\ell$ for $\ell\geq0$.


\section{Region $\Hcal_{\III}$: a non-Archimedean horseshoe map}\label{RegionIII}

\subsection{Overview}  Let $\{\pm\}^\ZZ$ be the set of bisequences
$$
s=(s_k)=(\dots s_{-3}s_{-2}s_{-1}.s_0s_1s_2s_3\dots)
$$
in the two symbols $+$ and $-$.  The set $\{\pm\}^\ZZ$ is naturally a compact topological space, endowed with the metric $d(s,s')=e^{-\min\{|k|\mid s_k\neq s'_k\}}$.  The {\em shift map} on $\{\pm\}^\ZZ$ is the homeomorphism $\sigma:\{\pm\}^\ZZ\to\{\pm\}^\ZZ$ defined by the rule $\sigma(s)_{k}=s_{k+1}$; in other words
$$
\sigma(\dots s_{-3}s_{-2}s_{-1}.s_0s_1s_2s_3\dots)=(\dots s_{-3}s_{-2}s_{-1}s_0.s_1s_2s_3\dots).
$$

The purpose of this section is to prove that, when $(a,b)$ is in the region $\Hcal_\III$ of the parameter space $\Hcal$ and $a$ is a square in $K$, the dynamical system obtained by restricting the H\'enon map $\phi_{a,b}:K^2\to K^2$ to its filled Julia set $J(\phi_{a,b})$ is topologically conjugate to the shift map $\sigma:\{\pm\}^\ZZ\to\{\pm\}^\ZZ$.  More precisely, there exists a homeomorphism $\omega:\{\pm\}^\ZZ\to J(\phi_{a,b})$ such that $\omega\circ\sigma=\phi_{a,b}\circ\omega$.

\begin{equation*}\label{TopConjDiagram}
\begin{CD}
\{\pm\}^\ZZ  @> \sigma >>   \{\pm\}^\ZZ \\ 
@V \omega VV                                    @VV \omega V \\ 
J(\phi_{a,b}) @> \phi_{a,b} >> J(\phi_{a,b})
\end{CD}
\end{equation*}


\subsection{Assumptions and definitions}  Throughout $\S$~\ref{RegionIII}, $\phi=\phi_{a,b}$ denotes a H\'enon map for $(a,b)$ in the region $\Hcal_\III$ of the parameter space $\Hcal$.  Thus $|a|>\max(1,|b|^2)$, and so in the notation of $\S$~\ref{FiltrationSect} we have $R=|a|^{1/2}$ and $S_R=B_{|a|^{1/2}}(0,0)$.  Since we know from Theorem~\ref{JuliaSetBoundedNonempty} part {\bf (d)} that $J(\phi_{a,b})$ is empty when $a$ is not a square in $K$, we also assume that $a$ is a square in $K$.  

To summarize, the following notation and assumptions are in effect throughout $\S$~\ref{RegionIII}:
\begin{itemize}
	\item $a,b\in K$, $b\neq0$, and $\phi=\phi_{a,b}:K^2\to K^2$ is defined as in $(\ref{HenonDef})$
	\item $a=\gamma^2$ for some $\gamma\in K$
	\item $|\gamma|>1$ and $|\gamma|>|b|$
	\item $I=\{x\in K\mid |x|\leq |\gamma|\}$
	\item $S=I\times I = \{(x,y)\in K^2\mid \|(x,y)\|\leq |\gamma|\}$
\end{itemize}


\subsection{Curves and tubes in $S$}   Given a function $f:I\to I$ and some $\delta>0$, define the sets  
\begin{equation*}
\begin{split}
V(f) & =\{(f(t),t)\in S\mid t\in I\} \\
V_\delta(f) & =\{(f(t)+\theta,t)\in S\mid t\in I, |\theta|\leq \delta\} \\
\end{split}
\end{equation*}
We call $V(f)$ the {\em vertical curve} in $S$ associated to $f$, and $V_\delta(f)$ the {\em vertical tube} of radius $\delta$ in $S$ associated to $f$.  Similarly, we define
\begin{equation*}
\begin{split}
H(f) & =\{(t,f(t))\in S\mid t\in I\} \\
H_\delta(f) & =\{(t,f(t)+\theta)\in S\mid t\in I, |\theta|\leq \delta\},
\end{split}
\end{equation*}
the {\em horizontal curve} in $S$ associated to $f$, and the {\em  horizontal tube} of radius $\delta$ in $S$ associated to $f$.

Given a set $D\subseteq K$, we say a function $f:D\to K$ is {\em $C$-Lipschitz} if $|f(t)-f(t')|\leq C|t-t'|$ for all $t,t'\in D$ and some constant $C>0$.  

The following lemma is a non-Archimedean analogue of \cite{MR1829194} Ch. 2 Lemma 2.

\begin{lem}\label{CurvesAndTubesLem} 
Let $f:I\to I$ be a $C_f$-Lipschitz function, let $g:I\to I$ be a $C_g$-Lipschitz function, and assume that $C_fC_g<1$.  Then $V(f)\cap H(g)$ contains exactly one point.
\end{lem}
\begin{proof}
The functions $f\circ g:I\to I$ and $g\circ f:I\to I$ are $C_fC_g$-Lipschitz, and so by our assumption that $C_fC_g<1$, it follows from the Banach fixed point theorem that they have unique fixed points, say $f(g(x_0))=x_0$ and $g(f(y_0))=y_0$, respectively.  Applying $f$ to the first equation we have $f(g(f(y_0)))=f(y_0)$ showing that $f(y_0)=x_0$ by uniqueness of $x_0$; similarly $y_0=g(x_0)$.  Since $(x_0,y_0)=(f(y_0),y_0)=(x_0,g(x_0))$, it is clear that $(x_0,y_0)$ is in both $V(f)$ and $H(g)$.  If $(x,y)$ is an arbitrary point in $V(f)\cap H(g)$, then $y=g(x)$ and $x=f(y)$, and hence $x$ is a fixed point of $f\circ g$ and $y$ is a fixed point of $g\circ f$.  It follows that $x=x_0$ and $y=y_0$, confirming uniqueness.
\end{proof}

\subsection{Horseshoe dynamics in $S$}   

The following lemma illustrates the characteristic ``horseshoe'' property of the H\'enon map; it states that the inverse image of a vertical tube meets $S$ at two thinner vertical tubes.  Following this lemma we deduce the analogous statement for horizontal tubes.

\begin{lem}\label{TwoVertTubesLem}
Let $f:I\to I$ be a $|b/\gamma|$-Lipschitz function and let $0<\rho\leq|\gamma|$.  There exist two $|b/\gamma|$-Lipschitz functions $f^\pm:I\to I$ such that $f^\pm(t)\in D_{|\gamma|}^\circ(\pm\gamma)$ for all $t\in I$ and
\begin{equation*}
\phi^{-1}(V_\rho(f))\cap S=V_{\rho/|\gamma|}(f^+)\cup V_{\rho/|\gamma|}(f^-).
\end{equation*}
\end{lem}
\begin{proof}
Let $r=\max(1,|b|)<|\gamma|$ and fix $t\in I$.  Define two $(r/|\gamma|)$-Lipschitz functions
\begin{equation*}
\begin{split}
& T_{t}^\pm:D_{r}(\pm\gamma)\to D_{r}(\pm\gamma) \\
& T_{t}^\pm(x) = x\pm\frac{1}{2\gamma}(a+bt-x^2-f(x)).
\end{split}
\end{equation*}
If $x=\gamma+\theta$ with $|\theta|\leq r$, then $|T_{t}^+(x)-\gamma|=|\frac{-\theta^2+bt-f(\gamma+\theta)}{2\gamma}|\leq r$, verifying that $T_{t}^+(D_{r}(\gamma))\subseteq D_{r}(\gamma)$.  To check the Lipschitz condition, note that for distinct $x_1,x_2\in D_{r}(\gamma)$ we have
$$
\bigg|\frac{T_{t}^+(x_1)-T_{t}^+(x_2)}{x_1-x_2}\bigg|=\frac{1}{|\gamma|}\bigg|2\gamma-(x_1+x_2)-\frac{f(x_1)-f(x_2)}{x_1-x_2}\bigg|\leq\frac{r}{|\gamma|}
$$
using the estimate $|2\gamma-(x_1+x_2)|\leq\max(|\gamma-x_1|,|\gamma-x_2|)\leq r$, and the assumption that $f$ is $|b/\gamma|$-Lipschitz and $|\frac{b}{\gamma}|<r$.  Similar calculations show that $T_{t}^-(D_{r}(-\gamma))\subseteq D_{r}(-\gamma)$ and that $T_{t}^-$ is $(r/|\gamma|)$-Lipschitz.

Since $T_{t}^+$ is contracting, it has a unique fixed point in $D_{r}(\gamma)$ by the Banach fixed-point theorem; call this point $f^+(t)$, and similarly define $f^-(t)\in D_{r}(\gamma)$ to be the unique fixed point of $T_{t}^-$.  We conclude that $f^\pm(t)\in D_{|\gamma|}^\circ(\pm\gamma)$ and 
\begin{equation}\label{FixedPointIdentity}
a+bt-f^\pm(t)^2-f(f^\pm(t))=0.
\end{equation}
We then have
\begin{equation}\label{TwoTubesIdentity}
\begin{split}
\phi^{-1}(V_\rho(f))\cap S & = \{(x,y)\in S\mid \phi(x,y)\in V_\rho(f)\} \\
	& = \{(x,y)\in S\mid |a+by-x^2-f(x)|\leq\rho\} \\
	& = V_{\rho/|\gamma|}(f^+)\cup V_{\rho/|\gamma|}(f^-).
\end{split}
\end{equation}

To see the last equality in $(\ref{TwoTubesIdentity})$, note that using $(\ref{FixedPointIdentity})$ we have
\begin{equation}\label{TwoTubesIdentity2}
\begin{split}
|a+by-x^2-f(x)| & = |a+by-x^2-f(x)-(a+by-f^\pm(y)^2-f(f^\pm(y)))| \\
	& = |(f^\pm(y)+x)(f^\pm(y)-x)+f(f^\pm(y))-f(x)|.
\end{split}
\end{equation}
If $(x,y)\in V_{\rho/|\gamma|}(f^+)$ then $|x-f^+(y)|\leq \rho/|\gamma|$.  It follows using $(\ref{TwoTubesIdentity2})$ that 
$$
|a+by-x^2-f(x)|\leq\rho
$$ 
because $|f^+(y)+x|\leq|\gamma|$ and 
\begin{equation}\label{TwoTubesIdentity3}
|f(f^+(y))-f(x)|\leq|b/\gamma||f^+(y)-x|\leq \rho|b|/|\gamma|^2<\rho;
\end{equation}
similarly if $(x,y)\in V_{\rho/|\gamma|}(f^-)$. 

Conversely, if $(x,y)\not\in V_{\rho/|\gamma|}(f^+)\cup V_{\rho/|\gamma|}(f^-)$, then both $R^\pm:=|x-f^\pm(y)|>\rho/|\gamma|$.  We observe that one of the two identities $|f^+(y)+x|=|\gamma|$ or $|f^-(y)+x|=|\gamma|$ must hold.  Otherwise, both $|f^\pm(y)+x|<|\gamma|$ hold, which implies $|f^+(y)-f^-(y)|<|\gamma|$.  But then 
$$
|f^+(y)-f^-(y)|=|2\gamma+(f^+(y)-\gamma)-(f^-(y)+\gamma)|=|2\gamma|=|\gamma|,
$$ 
a contradiction.  If $|f^+(y)+x|=|\gamma|$, we deduce using $(\ref{TwoTubesIdentity2})$ and $(\ref{TwoTubesIdentity3})$ that
\begin{equation*}
|a+by-x^2-f(x)| = |\gamma|R^+>\rho,
\end{equation*}
and similarly if $|f^-(y)+x|=|\gamma|$.

It remains to show that $f^\pm$ are Lipschitz.  For distinct $t_1,t_2\in I$, set $u_i=f^\pm(t_i)$.  Thus $u_1,u_2\in D^\circ_{|\gamma|}(\pm\gamma)$, so $|u_1+u_2|=|\pm2\gamma|=|\gamma|$ and hence
$$
|f(u_1)-f(u_2)|/|u_1-u_2|\leq|b/\gamma|<1<|\gamma|.
$$
Using $(\ref{FixedPointIdentity})$ we deduce
\begin{equation*}
\frac{|b||t_1-t_2|}{|u_1-u_2|} = \bigg|u_1+u_2+\frac{f(u_1)-f(u_2)}{u_1-u_2}\bigg|=|\gamma|
\end{equation*}
and thus $|f^\pm(t_1)-f^\pm(t_2)|=|b/\gamma||t_1-t_2|$.
\end{proof}

\begin{lem}\label{TwoHorTubesLem}
Let $g:I\to I$ be a $|1/\gamma|$-Lipschitz function and let $0<\rho\leq|\gamma|$.  There exist two $|1/\gamma|$-Lipschitz functions $g^\pm:I\to I$ such that $g^\pm(t)\in D_{|\gamma|}^\circ(\pm\gamma)$ for all $t\in I$ and
\begin{equation}\label{TwoHorTubes}
\phi(H_\rho(g))\cap S=H_{\rho|b|/|\gamma|}(g^+)\cup H_{\rho|b|/|\gamma|}(g^-).
\end{equation}
\end{lem}
\begin{proof}
We conjugate by the automorphism $\lambda(x,y)=(-by,-bx)$ and use the involution $\iota:\Hcal\to\Hcal$ described in Proposition~\ref{conj winv}.  Set $\gamma_*=-\gamma/b$, $I_*=\{x\in K\mid |x|\leq |\gamma_*|\}$, and $S_*=I_* \times I_*$.  Thus $\lambda(S_*)=S$, and $\lambda$ takes vertical tubes in $S_*$ to horizontal tubes in $S$.  More precisely, considering the $|b/\gamma_*|$-Lipschitz function $f_*:I_*\to I_*$ defined by $f_*(t)=-\frac{1}{b}g(-bt)$, we have
\begin{equation}\label{TwoHorTubes2}
\lambda(V_\delta(f_*))=H_{|b|\delta}(g).
\end{equation}

We now apply Lemma~\ref{TwoVertTubesLem} to $f_*$ and the H\'enon map $\phi_{\iota(a,b)}$, with $\rho_*=\rho/|b|$.  We obtain
$$
\phi_{\iota(a,b)}^{-1}(V_{\rho_*}(f_*))\cap S_*=V_{\rho_*/|\gamma_*|}(f_*^+)\cup V_{\rho_*/|\gamma_*|}(f_*^-)
$$
for $|b/\gamma_*|$-Lipschitz functions $f_*^\pm:I_*\to I_*$ such that $f^\pm_*(t)\in D_{|\gamma_*|}^\circ(\pm\gamma_*)$ for all $t\in I_*$.  Applying $\lambda$ to this identity we obtain
\begin{equation}\label{TwoHorTubes3}
(\lambda\circ\phi_{\iota(a,b)}^{-1}\circ\lambda^{-1})(\lambda(V_{\rho_*}(f_*)))\cap \lambda(S_*)=\lambda(V_{\rho_*/|\gamma_*|}(f_*^+))\cup \lambda(V_{\rho_*/|\gamma_*|}(f_*^-)).
\end{equation}
By Proposition~\ref{conj winv}, $\phi=\lambda\circ\phi_{\iota(a,b)}^{-1}\circ\lambda^{-1}$, thus the desired identity $(\ref{TwoHorTubes})$ follows from $(\ref{TwoHorTubes2})$ and $(\ref{TwoHorTubes3})$, with $g^\pm:I\to I$ defined by $g^\pm(t)=-bf^\pm_*(\frac{-t}{b})$.
\end{proof}

The next lemma describes the forward filled Julia set in $S$ as an uncountable union of vertical curves, indexed by the set of sequences in two symbols.  We then prove the analogous statement describing the backward filled Julia set in terms of horizontal curves.

\begin{lem}\label{ForwardTrajectoryLemma}
There exists a family of $|b/\gamma|$-Lipschitz functions $f^s:I\to I$, indexed by the set of all sequences $s=(s_0s_1s_2\dots)$, where each $s_i\in\{\pm\}$, and a pair $V^\pm$ of disjoint subsets of $S$ such that
\begin{equation}
V(f^{s})=\{(x,y)\in S\mid \phi^k(x,y)\in V^{s_k}\text{ for all }k\geq0\}.
\end{equation}
Moreover, $J^+(\phi)\cap S=\bigcup_{s=(s_0s_1s_2\dots)}V(f^s)$.
\end{lem}
\begin{proof}
To ease notation set $\delta_n=1/|\gamma|^n$.  We first construct a family of $|b/\gamma|$-Lipschitz functions $f^s_{n}:I\to I$, indexed by sequences $s=(s_0s_1s_2\dots)$ in the two symbols $\{\pm\}$ and integers $n\geq0$.  When $n=0$, we apply Lemma~\ref{TwoVertTubesLem} with $f:I\to I$ equal to the identically zero function and $\rho=|\gamma|$; thus $V_{|\gamma|}(f)=S$, and we obtain $|b/\gamma|$-Lipschitz functions $f^\pm:I\to I$ with $f^\pm(t)\in D_{|\gamma|}^\circ(\pm\gamma)$.  Define $V^\pm=V_{1}(f^\pm)$ and $f^s_0=f^{s_0}$.  We then have
\begin{equation}\label{FirstTwoVerticalTubes}
\phi^{-1}(S)\cap S=V^+\cup V^-.
\end{equation}

Fix $n\geq0$ and assume that the functions $f^s_{n}:I\to I$ have been constructed for all sequences $s=(s_0s_1s_2\dots)$ in the two symbols $\{\pm\}$.  We apply Lemma~\ref{TwoVertTubesLem} with $f=f_n^{\sigma(s)}$ and $\rho=\delta_n=1/|\gamma|^{n}$, where $\sigma(s_0s_1s_2\dots)=(s_1s_2s_3\dots)$.  We obtain $|b/\gamma|$-Lipschitz functions $f^\pm:I\to I$ with $f^\pm(t)\in D_{|\gamma|}^\circ(\pm\gamma)$.  Set $f^s_{n+1}=f^{s_0}$.  We then have 
\begin{equation}\label{GeneralTwoVerticalTubes}
\phi^{-1}(V_{\delta_n}(f_n^{\sigma(s)}))\cap S=V_{\delta_{n+1}}(f^s_{n+1})\cup V_{\delta_{n+1}}(f^{s'}_{n+1})
\end{equation}
where $s'$ is obtained from $s=(s_0s_1s_2\dots)$ by changing $s_0$ from $\pm$ to $\mp$.

When $n\geq1$, using $(\ref{GeneralTwoVerticalTubes})$, we see that a point $(x,y)$ is in $V_{\delta_{n}}(f^s_{n})$ if and only if $(x,y)\in V^{s_0}$ and $\phi(x,y)\in V_{\delta_{n-1}}(f_{n-1}^{\sigma(s)})$.  It follows from this and a simple induction that 
\begin{equation}\label{TrajectorySetsAreTubes}
V_{\delta_{n}}(f^{s}_{n})=\{(x,y)\in S\mid \phi^k(x,y)\in V^{s_k}\text{ for all }0\leq k\leq n\}.
\end{equation}
In other words, $V_{\delta_{n}}(f^{s}_{n})$ is the set of points in $S$ whose partial forward orbit follows a particular trajectory through the two disjoint sets $V^\pm$.  

From $(\ref{TrajectorySetsAreTubes})$ it is clear that
$$
V_{\delta_{n+1}}(f^{s}_{n+1})\subseteq V_{\delta_{n}}(f^{s}_{n}),
$$
from which it follows that the limit $f^s(t) := \lim_{n\to+\infty}f^s_{n}(t)$ exists, and a standard argument shows that a limit of $|b/\gamma|$-Lipschitz functions is $|b/\gamma|$-Lipschitz.  We conclude using $(\ref{TrajectorySetsAreTubes})$ that
\begin{equation}\label{CompleteTrajectorySetsAreCurves}
\begin{split}
V(f^{s})=\bigcap_{n\geq0}V_{\delta_{n}}(f^{s}_{n})=\{(x,y)\in S\mid \phi^k(x,y)\in V^{s_k}\text{ for all }k\geq0\}.
\end{split}
\end{equation}

It follows from Proposition~\ref{Filtration} that $J^+(\phi)\cap S$ is the set of points in $S$ whose forward orbit is contained in $S$.  Using $(\ref{FirstTwoVerticalTubes})$ we have 
$$
J^+(\phi)\cap S\subseteq \phi^{-1}(S)\cap S=V^{+}\cup V^-.
$$
Thus every point in $J^+(\phi)\cap S$ has a forward orbit which follows some trajectory through the two sets $V^\pm$, and hence $J^+(\phi)\cap S\subseteq\bigcup_{s=(s_0s_1s_2\dots)}V(f^s)$ using $(\ref{CompleteTrajectorySetsAreCurves})$.  Conversely, $(\ref{CompleteTrajectorySetsAreCurves})$ also shows that any point in some $V(f^s)$ has forward orbit contained in $S$ and hence is in $J^+(\phi)\cap S$.
\end{proof}

\begin{lem}\label{BackwardTrajectoryLemma}
There exists a family of $|1/\gamma|$-Lipschitz functions $g^s:I\to I$, indexed by the set of all sequences $s=(\dots s_{-3}s_{-2}s_{-1})$, where each $s_i\in\{\pm\}$, and a pair $H^\pm$ of disjoint subsets of $S$ such that
\begin{equation}\label{CompleteHorTrajectorySetsAreCurves}
H(g^{s})=\{(x,y)\in S\mid \phi^{k+1}(x,y)\in H^{s_{k}}\text{ for all }k\leq -1\}.
\end{equation}
Moreover, $J^-(\phi)\cap S=\bigcup_{s=(\dots s_{-3}s_{-2}s_{-1})}H(g^s)$.
\end{lem}
\begin{proof}
This proof differs from Lemma~\ref{ForwardTrajectoryLemma} only in technical details, which we summarize.  To ease notation set $\epsilon_{n}=|b|^{n+1}/|\gamma|^n$.  We first construct a family of $|1/\gamma|$-Lipschitz functions $g^s_{n}:I\to I$, indexed by the set of sequences $s=(\dots s_{-3}s_{-2}s_{-1})$ in the two symbols $\{\pm\}$ and integers $n\geq0$.  When $n=0$, we apply Lemma~\ref{TwoHorTubesLem} with $g:I\to I$ equal to the identically zero function and $\rho=|\gamma|$.  We obtain $|1/\gamma|$-Lipschitz functions $g^\pm:I\to I$ with $g^\pm(t)\in D_{|\gamma|}^\circ(\pm\gamma)$.  Define $H^\pm=H_{|b|}(g^\pm)$ and $g^s_{0}=g^{s_{-1}}$.  We then have $\phi(S)\cap S=H^+\cup H^-$.

Fixing $n\geq0$ and assuming that the functions $g^s_{n}:I\to I$ have been constructed for all sequences $s=(\dots s_{-3}s_{-2}s_{-1})$ in the two symbols $\{\pm\}$, we apply Lemma~\ref{TwoHorTubesLem} with $g=g_n^{\sigma^{-1}(s)}$ and $\rho=\epsilon_{n}=|b|^{n+1}/|\gamma|^n$, where $\sigma^{-1}(\dots s_{-3}s_{-2}s_{-1})=(\dots s_{-4}s_{-3}s_{-2})$.  We obtain $|1/\gamma|$-Lipschitz functions $g^\pm:I\to I$ with $g^\pm(t)\in D_{|\gamma|}^\circ(\pm\gamma)$.  Setting $g^s_{n+1}=g^{s_{-1}}$, we have
\begin{equation}\label{GeneralTwoHorTubes}
\phi(H_{\epsilon_{n}}(g_n^{\sigma^{-1}(s)}))\cap S=H_{\epsilon_{n+1}}(g^s_{n+1})\cup H_{\epsilon_{n+1}}(g^{s'}_{n+1})
\end{equation}
where $s'$ is obtained from $s=(\dots s_{-3}s_{-2}s_{-1})$ by changing $s_{-1}$ from $\pm$ to $\mp$.

An induction argument using $(\ref{GeneralTwoHorTubes})$ shows that 
\begin{equation}\label{HorTrajectorySetsAreTubes}
H_{\epsilon_{n}}(g^{s}_{n})=\{(x,y)\in S\mid \phi^{k+1}(x,y)\in H^{s_{k}}\text{ for all }-(n+1)\leq k\leq -1\},
\end{equation}
from which it follows that $H_{\epsilon_{n+1}}(g^{s}_{n+1})\subseteq H_{\epsilon_{n}}(g^{s}_{n})$.  The limit $g^s(t) = \lim_{n\to+\infty}g^s_{n}(t)$ exists and is $|1/\gamma|$-Lipschitz, and since $H(g^{s})=\bigcap_{n\geq0}H_{\epsilon_{n}}(g^{s}_{n})$ we obtain $(\ref{CompleteHorTrajectorySetsAreCurves})$.  The proof that $J^-(\phi)\cap S=\bigcup_{s=(\dots s_{-3}s_{-2}s_{-1})}H(g^s)$ follows from the same argument used in Lemma~\ref{ForwardTrajectoryLemma}.
\end{proof}

\begin{rem}\label{HandVRemark}
It will be necessary in the proof of Theorem~\ref{TopologicalConjThm} to observe the following relationship between the sets $V^\pm$ and $H^\pm$ occurring in Lemma~\ref{ForwardTrajectoryLemma} and Lemma~\ref{BackwardTrajectoryLemma}.  Since $\phi^{-1}(S)\cap S=V^+\cup V^-$ and $\phi(S)\cap S=H^+\cup H^-$, we obtain $\phi(V^+)\cup \phi(V^-)=H^+\cup H^-$.  By construction, $H^\pm\subseteq I\times D_{|\gamma|}^\circ(\pm\gamma)$ and $V^\pm\subseteq D_{|\gamma|}^\circ(\pm\gamma)\times I$, and it follows easily that $\phi(V^\pm)\subseteq I\times D_{|\gamma|}^\circ(\pm\gamma)$ as well.  As $I\times D_{|\gamma|}^\circ(\pm\gamma)$ are disjoint, we conclude $\phi(V^\pm)=H^\pm$.
\end{rem}

\subsection{The topological conjugacy to the shift map}\label{TopConSect}  We are ready to prove the main result of this section, a non-Archimedean analogue of a theorem on the real H\'enon map due to Devaney-Nitecki \cite{MR539548}.

\begin{thm}\label{TopologicalConjThm}
Let $(a,b)\in\Hcal_\III$, suppose that $a$ is a square in $K$, and let $\phi=\phi_{a,b}:K^2\to K^2$ be the associated H\'enon map.   There exists a homeomorphism $\omega:\{\pm\}^\ZZ\to J(\phi)$ such that $\omega\circ\sigma=\phi\circ\omega$.
\end{thm}

\begin{proof}
Given $s=(\dots s_{-2}s_{-1}s_0s_1s_2\dots)\in\{\pm\}^\ZZ$, since the functions $f^{(s_0s_1s_2\dots)}:I\to I$ and $g^{(\dots s_{-3}s_{-2}s_{-1})}:I\to I$ are $|b/\gamma|$-Lipschitz and $|1/\gamma|$-Lipschitz, respectively, and $|b/\gamma||1/\gamma|=|b|/|\gamma|^2<1$, Lemma~\ref{CurvesAndTubesLem} implies that 
$$
H(g^{(\dots s_{-3}s_{-2}s_{-1})})\cap V(f^{(s_0s_1s_2\dots)})=\{\omega(s)\}
$$ 
for some point $\omega(s)\in S$.  This defines a function $\omega:\{\pm\}^\ZZ\to J(\phi)$ since, by Lemma~\ref{ForwardTrajectoryLemma} and Lemma~\ref{BackwardTrajectoryLemma}, $\omega(s)\in(J^-(\phi)\cap S)\cap(J^+(\phi) \cap S)=J(\phi)$. 

If $(x,y)\in J(\phi)$, then $(\ref{FirstTwoVerticalTubes})$ implies that every point in its orbit is contained in one of the two sets $V^\pm$, and similarly, every point in its orbit is contained in one of the two sets $H^\pm$.  Define $s=(s_k)\in\{\pm\}^\ZZ$ by $\phi^k(x,y)\in V^{s_k}$ for $k\geq0$, and $\phi^{k+1}(x,y)\in H^{s_{k}}$ for $k\leq-1$.  The function $(x,y)\mapsto s$ defines an inverse $\omega^{-1}:J(\phi)\to\{\pm\}^\ZZ$ by Lemma~\ref{ForwardTrajectoryLemma} and Lemma~\ref{BackwardTrajectoryLemma}, and so $\omega$ is bijective.

A neighborhood base for the topology on $\{\pm\}^\ZZ$ is composed of cylinder sets
$$
\Sigma_{t_{-N},\dots,t_N}=\{s\in \{\pm\}^\ZZ\mid s_k=t_k\text{ for all }|k|\leq N\}.
$$
It follows from $(\ref{TrajectorySetsAreTubes})$ and $(\ref{HorTrajectorySetsAreTubes})$ that $\omega(\Sigma_{t_{-N},\dots,t_N})$ is the intersection of a vertical tube and a horizontal tube of positive radii.  As tubes of positive radii are topologically open, this shows that $\omega^{-1}$ is continuous.  A standard exercise states that a continuous bijection of compact sets has a continuous inverse, and thus $\omega$ is a homeomorphism.

It remains to show that $\omega\circ\sigma=\phi\circ\omega$.  Fix $s\in\{\pm\}^\ZZ$ and let $t=\sigma(s)$; thus $t_k=s_{k+1}$.  Let $(x,y)=\omega(s)$.  By Lemma~\ref{ForwardTrajectoryLemma} and Lemma~\ref{BackwardTrajectoryLemma}, we have $\phi^k(x,y)\in V^{s_k}$ for all $k\geq0$ and $\phi^{k+1}(x,y)\in H^{s_{k}}$ for all $k\leq-1$.  Thus $\phi^k(\phi(x,y))\in V^{t_k}$ for all $k\geq0$, and $\phi^{k+1}(\phi(x,y))\in H^{t_{k}}$ for all $k\leq-2$.  Using Remark~\ref{HandVRemark}, we also have $\phi(x,y)\in\phi(V^{s_0})=H^{s_0}=H^{t_{-1}}$.  By Lemma~\ref{ForwardTrajectoryLemma} and Lemma~\ref{BackwardTrajectoryLemma} we conclude 
$$
\phi(x,y)\in H(g^{(\dots t_{-3}t_{-2}t_{-1})})\cap V(f^{(t_0t_1t_2\dots)})
$$ 
and therefore $\phi(x,y)=\omega(t)$; in other words, $\phi(\omega(s))=\omega(\sigma(s))$.
\end{proof}

\begin{rem}
We can now show that the forward and backward filled Julia sets  $J^\pm(\phi)$ are unbounded, finishing the proof of Proposition~\ref{JuliaSetBoundedNonempty} part {\bf (c)}.  As described in Remark~\ref{HandVRemark}, we have 
$$
J(\phi)\subseteq V^+\cup V^-\subseteq(D_{|\gamma|}^\circ(\gamma)\times I)\cup(D_{|\gamma|}^\circ(-\gamma)\times I)
$$ 
In particular, since $0\not\in D_{|\gamma|}^\circ(\pm\gamma)$, given any $s\in\{\pm\}^\ZZ$, the point $(0,g^s(0))$ is in $H(g^s)\subseteq J^-(\phi)$ but is not in $J(\phi)$.  As $J^-(\phi)$ is $\phi$-invariant, the entire orbit of $(0,g^s(0))$ is contained in $J^-(\phi)$, and since this orbit is unbounded, $J^-(\phi)$ is unbounded; similarly, $J^+(\phi)$ is unbounded.
\end{rem}



\def\cprime{$'$}

\end{document}